\documentclass[UTF-8,reqno]{amsart}

\usepackage{CJK}

\usepackage{amsmath, amsfonts, amsthm, amssymb, stmaryrd, color, cite}

\usepackage{esint}
\usepackage{hyperref}

\newtheorem{theorem}{Theorem}[section]
\newtheorem{lemma}{Lemma}[section]
\newtheorem{proposition}{Proposition}[section]

\theoremstyle{definition}
\newtheorem{definition}{Definition}[section]
\numberwithin{equation}{section}
\theoremstyle{remark}

\allowdisplaybreaks

\begin{document}
\begin{CJK*}{GBK}{song}

\title[The Landau-Lifshitz-Bloch equation on the thin film]
{The Landau-Lifshitz-Bloch equation on the thin film}

\author[Y. He]{Yuxun He}
\address{College of Mathematics and Statistics, Chongqing University, Chongqing, 401331, China.}
\email{hyuxun@163.com}

\author[H. Wang]{Huaqiao Wang}
\address{College of Mathematics and Statistics, Chongqing University, Chongqing, 401331, China.}
\email{wanghuaqiao@cqu.edu.cn}

\keywords{micromagnetics, Landau-Lifshitz-Bloch equation, ferromagnetic film, stray field, weak solution, limit equation}

\date{\today}

\begin{abstract}
We consider the initial boundary value problem of Landau-Lifshitz-Bloch equation on three-dimensional ferromagnetic films, where the effective field contains the stray field controlled by Maxwell equation and the exchange field contains exchange constant. In this paper, we establish the existence of weak solutions of the equation by using the Faedo-Galerkin approximation method.  We also derive its two-dimensional limit equation in a mathematically rigorous way when the film thickness tends to zero under appropriate compactness conditions. Moreover, we obtain an equation that can better describe the magnetic dynamic behavior of ferromagnetic films with negligible thickness at high temperature.
\end{abstract}

\maketitle

\section{Introduction}
In 1935, Landau and Lifshitz derived the Landau-Lifshitz (LL) equation \cite{LDE,FA,MC,GP} to describe the evolution of spin field in continuous ferromagnetic field below the critical (Curie) temperature. LL equation is the cornerstone of the dynamic magnetization theory of ferromagnetic materials, and its micromagnetic method is also the basis of most theoretical studies of thermal magnetization dynamics. The most important feature is that the magnetization is constant. However, the damping term in LL equation is only applicable to the case of small damping, and it is no longer applicable when encountering large damping. Glibert deduced the Landau-Lifshitz-Gilbert (LLG) equation \cite{GT,KP,MT,CI} based on the LL equation, which can be used in the case of large damping. It can be transformed into LL equation, which is mathematically identical. The motion of domain wall is the basic mechanism of ferromagnetic mode dynamics. The internal structure of domain wall in thin films and its influence on the formation of magnetic mode have always been a topic of general interest. According to the magnetization dynamics described by LL equation or LLG equation, physicists and mathematicians \cite{GJ,KS,DKM,HT} have done a lot of research on the asymptotic behavior of the magnetization limit of thin films under different parameter mechanisms by the energy method. Based on the limit asymptotic behavior of magnetization distribution on the film, they also discussed the limit asymptotic behavior of weak solutions of LL equation or LLG equation, and deduced the limit equation on the film \cite{WGC,MR,CMO,CM}.

In recent years, the research on magnetic materials in the field of thermal excitation has become particularly important. So more and more studies focus on the dynamic behavior of ferromagnetic materials at high temperature. Although micromagnetic techniques based on LL equation or LLG equation can perfectly describe the micro nano scale magnetic film system at low temperature, they can not accurately describe the magnetodynamic behavior at high temperature (especially close to the Curie temperature of the material). Therefore, from the perspective of magnetization dynamic modeling, micromagnetism theory needs to be further developed. In 1997, Garanin \cite{GDA1,GDA2} proposed an effective Landau-Lifshitz-Bloch (LLB) equation at high temperature, especially when it is close to Curie temperature $T_{C}$ and ultrafast time scale, as follows:
\begin{equation}\label{1.1}
\frac{\partial\textbf{u}}{\partial t}=\gamma\textbf{u}\times\textbf{H}_{eff}+L_{1}\frac{1}{|\textbf{u}|^{2}}\left(\textbf{u}\cdot\textbf{H}_{eff}\right)\textbf{u}-L_{2}
\frac{1}{|\textbf{u}|^{2}}\textbf{u}\times\left(\textbf{u}\times\textbf{H}_{eff}\right),
\end{equation}
where the spin polarization $\textbf{u}(\textbf{x},t)=\frac{\textbf{m}}{m^{0}_{s}}$($t>0$, $\textbf{x}\in \mathcal{B}\subset\mathbb{R}^{3}$), $\textbf{m}$ denotes magnetization, $m^{0}_{s}$ is the saturation magnetization value when $T=0$; $\gamma>0$ stands for the gyromagnetic ratio; $L_{1}$ and $L_{2}$ are the longitudinal and transverse damping coefficients, respectively. Effective field $\textbf{H}_{eff}$ in this equation is different from LL equation or LLG equation. In form, there will be a term related to longitudinal susceptibility, temperature and spin polarization. LLB equation successfully connects ferromagnetism with thermodynamic properties. One of its important properties is that magnetization is no longer conserved, but a dynamic variable.
Le \cite{KNL} considered a definite form of LLB equation\eqref{1.1}, and sets $T\geq T_{c}$, then $L_{1}=L_{2}$. The effective field is $\textbf{H}_{eff}=\Delta\textbf{u}-\frac{1}{\chi_{11}}\left(1+\frac{3}{5}\frac{T}{T-T_{c}}|\textbf{u}|^{2}\right)\textbf{u},$
where $\chi_{11}$ is the longitudinal magnetic susceptibility.
He proved the global existence of weak solutions of the LLB equation by using Faedo-Galerkin approximation method.
Inspired by him, Jia \cite{ZJ} proved the local existence of strong solutions. Guo et al. \cite{GL} further proved the global existence of smooth solutions. After that, Guo et al. \cite{GBL} also proved the global existence of weak solutions and smooth solutions of Landau-Lifshitz-Bloch-Maxwell equation which coupled by LLB equation and Maxwell equation system.

From the perspective of science and technology, the dynamic behavior of magnetization distribution on ferromagnetic thin films is an interesting and important problem. The dynamic problem of LLB equation on thin films also has important physical significance and application value. Physicists have done a lot of research on the related problems of magnetic films by LLB equation, such as using LLB equation to simulate the magnetization dynamics of a single crystal in the film, and then studying the related physical processes \cite{LARE}. In addition, LLB equation also has a wide range of practical applications in thin films, including describing the ultrafast demagnetization and subsequent recovery process caused by pulsed laser irradiation of magnetic thin films \cite{SA}, studying the spin-Seebeck effect (SSE) \cite{HN} found in the thin films, so as to help the development of new spin thermoelectronic devices, and studying heat-assisted magnetic recording (HAMR) technology to improve the storage density of storage media such as ferromagnetic thin films \cite{TWM}.

However, it is different from the LL equation or LLG equation which has a large number of results in mathematics. At present, there is little research about LLB equation on thin films in mathematics, especially there is no accurate derivation and proof of its limit equation.
Le \cite{KNL} has proved the existence of weak solutions of LLB equation considering only exchange field. In this paper, we couple the simplified Maxwell equation with LLB equation, add stray field term to the effective field, and prove the existence of weak solutions of LLB equation in this case. At the same time, we establish the three-dimensional model of LLB equation on thin film and continue the parameter setting when Melcher \cite{CM} studied the limit problem of LLG equation, that is, considering the case of medium time scale and small damping. Based on Le \cite{KNL}, we add the self induced magnetic field (stray field) of ferromagnet, which is controlled by simplified Maxwell equation \cite{JDJ}. We consider the LLB equation \eqref{1.1} on the film $\Omega(h):\Omega\times(0,h)(\Omega\in\mathbb{R}^{2})$, assuming $T\geq T_{c}$, then the longitudinal damping coefficient $L_{1}$ and lateral damping coefficient $L_{2}$ equals. The effective field $\textbf{H}_{eff}=A\Delta\textbf{u}-\frac{1}{\chi_{11}}\left(1+\frac{3}{5}\frac{T}{T-T_{c}}|\textbf{u}|^{2}\right)\textbf{u}-\nabla U,$
where $A$ is the exchange constant, $\chi_{11}$ is the longitudinal susceptibility, and the gradient field $-\nabla U$ is the stray field induced by spin polarization.
The spin polarization $\textbf{u}=\frac{\textbf{m}}{m^{0}_{s}}$ has the following relationship with its corresponding stray field potential \cite{CJGC}:
$$\Delta U=div\left(\textbf{u}\chi_{\Omega(h)}\right).$$
For $\textbf{u}\in H^{1}(\Omega(h))$, $U\in\dot{H}^{1}(\mathbb{R}^{3})$, in weak form this equation reads
\begin{equation}\label{1.2}
\int_{\mathbb{R}^{3}}\nabla U\cdot\nabla\varphi dx=\int_{\Omega(h)}\textbf{u}\cdot\nabla\varphi dx,\;\forall\varphi\in C_{0}^{\infty}(\mathbb{R}^{3}).\\
\end{equation}
And the following inequality \cite{GP,LXG} holds:
\begin{equation}\label{1.3}
\|\nabla U\|_{L^{p}(\mathbb{R}^{3})}\leq C\|\textbf{u}\|_{L^{p}(\Omega(h))},\quad 1<p<\infty.\\
\end{equation}
For convenience of representation, we define the parameter symbols as follows:
$$L:=L_{1}=L_{2},\quad\mu:=\frac{3T}{5(T-T_{c})},\quad \textbf{H}:=\textbf{H}_{eff}.$$
From $a\times(b\times c)=b(a\cdot c)-c(a\cdot b)$, we have
$$\textbf{u}\times(\textbf{u}\times\textbf{H})=(\textbf{u}\cdot\textbf{H})\textbf{u}-|\textbf{u}|^{2}\textbf{H}.$$
So equation \eqref{1.1} becomes
\begin{equation}\label{1.10}
\frac{\partial\textbf{u}}{\partial t}=\gamma\textbf{u}\times\textbf{H}+L\textbf{H}=L(A\Delta\textbf{u}-\nabla U)+\gamma \textbf{u}\times(A\Delta\textbf{u}-\nabla U)-\frac{L}{\chi_{11}}\left(1+\mu|\textbf{u}|^{2}\right)\textbf{u}.
\end{equation}

We proceed to adopt the dynamic mechanism in \cite{CM}, assuming that energy  is dominated by exchange energy, and considering the case of medium time scale and small damping, that is, for some $\varepsilon>0$ and $a>0$
$$\frac{A(h)}{h}\rightarrow\varepsilon,\quad\gamma(h)\sqrt{h}\rightarrow1,\quad\frac{L(h)}{\gamma(h)h}\rightarrow a, \text{ as } h\rightarrow0.$$
The leading order energy effect of the stray field interaction in the film is a quadratic shape anisotropy conducive to in-plane magnetization, resulting in the formation of a forcing term pointing to the film plane and competing with the cyclotron force pushing the magnetization vector away from the plane \cite{CM}. Therefore, in the limit process ($h\rightarrow 0$), some stray field energy will be converted into kinetic energy, and the limit spin polarization $u$ is planar.

LLB equation is developed from LL equation and LLG equation, but its form is more complex. The nonlinear term is increased from only involving the multiplication of two terms to three terms, which increases the influence on the interaction of stray field. More importantly, the favorable condition that the magnetization is constant is lost. First, the most direct impact is that the space of $\textbf{u}_{t}$ is worse than $L^{2}$ space, and some existing methods for $L^{2}$ estimation can no longer be used. Secondly, when considering two important physical quantities in the vertical direction, if we still set $w_{1}^{h}=\frac{1}{\sqrt{h}}\fint_{0}^{h}u_{3}dx_{3}$ and $w_{2}^{h}=\frac{1}{\sqrt{h}}\fint_{0}^{h}\frac{\partial U}{\partial x_{3}}dx_{3}$, then the final limit equation will show singularity. So we have developed and improved some new methods and techniques. We add the square of spin polarization to $w_{1}^{h}$ and $w_{2}^{h}$ to avoid singularity, that is $w_{1}^{h}=\frac{1}{\sqrt{h}}\fint_{0}^{h}\left|\textbf{u}^{2}\right|u_{3}dx_{3}$, $w_{2}^{h}=\frac{1}{\sqrt{h}}\fint_{0}^{h}\left|\textbf{u}^{2}\right|\frac{\partial U}{\partial x_{3}}dx_{3}$. This improvement produces many nonlinear terms involving the multiplication of three terms, which not only requires more and more precise estimates except $L^{2}$, but also requires higher space for strong convergence. In order to solve these problems, we extend the $L^{2}-L^{2}$ estimate on the average of product integrals to $L^{p}-L^{q}$ estimate. And then we use $L^{p}-L^{q}$ estimate, the energy methods and Gagliardo-Nirenberg interpolation to obtain more precise estimates of magnetization and stray field potential. Because the strong convergence space obtained by directly using the Compact Embedding Theorem in the past is poor, we combine the weak compactness argument with the Aubin-Lions Lemma to obtain the strong convergence in the better space. Finally, we derive the limit equation of LLB equation under the condition of appropriate compactness, and  prove the limit process in a mathematically rigorous way.


Now, we formally derive the limit equation of equation \eqref{1.10} when $h\rightarrow 0$. Assume that
$$\frac{\partial\textbf{u}^{h}}{\partial t}=\gamma\textbf{u}^{h}\times\textbf{H}+L\textbf{H}=:\gamma\textbf{u}^{h}\times\bar{\textbf{H}}+L\textbf{H},$$
where $\bar{\textbf{H}}=A\Delta\textbf{u}^{h}-\nabla U $, this effective field is only composed of exchange field and stray field.
The energy of effective field $\bar{\textbf{H}}$ is
$$\bar{E}(\textbf{u}^{h})=\frac{A}{2}\int_{\Omega(h)}|\nabla\textbf{u}^{h}|^{2}dx+\frac{1}{2}\int_{\mathbb{R}^{3}}|\nabla U|^{2}dx.$$
For renormalized $\bar{\textbf{H}}^{h}=\frac{\bar{\textbf{H}}}{h}$, the renormalized LLB equation reads
\begin{align}\label{1.5}
\frac{\partial\textbf{u}^{h}}{\partial t}=\gamma h\textbf{u}^{h}\times\bar{\textbf{H}}^{h}+L\textbf{H}.
\end{align}
The renormalized stray field potential is $v^{h}=\frac{U}{h},$ then the renormalized energy is given by
$$\bar{E}_{h}(\textbf{u}^{h})=\frac{1}{h^{2}}\bar{E}(\textbf{u}^{h})=\frac{A}{2h}\fint_{\Omega(h)}|\nabla\textbf{u}^{h}|^{2}dx+\frac{1}{2}
\int_{\mathbb{R}^{3}}\left|\nabla v^{h}\right|^{2}dx.$$
We temporarily assume that when $h\rightarrow0~(x_{3}\rightarrow0)$, $\fint_{0}^{h}\textbf{u}^{h}dx_{3}:=(\fint_{0}^{h}u^{h}dx_{3},\fint_{0}^{h}u_{3}^{h}dx_{3})\rightarrow(u,0)$, $\bar{\textbf{H}}^{h}:=(\bar{H}^{h},\bar{H}_{3}^{h})\rightarrow \bar{\textbf{H}}_{0}=(\bar{H}_{0},\bar{H}_{0,3})$, the spin polarization of the limit field and its self induced magnetic field are $u$, $v$, respectively. Since the limit equations of the first component and the second component are always valid in form, we get the result by considering the third component of \eqref{1.5} in a weak sense
\begin{align*}
\int_{0}^{T}\int_{\Omega(h)}\frac{\partial u_{3}^{h}}{\partial t}\cdot\phi dxdt=\int_{0}^{T}\int_{\Omega(h)}\left[\gamma h\left(u_{1}^{h}\bar{H}^{h}_{2}-u_{2}^{h}\bar{H}^{h}_{1}\right)+LH_{3}\right]\cdot\phi dxdt,
\end{align*}
where $\phi$ is a test function. When $h\rightarrow0$, from the above formula we obtain
\begin{align*}
\int_{0}^{T}\int_{\Omega(h)}\left(u^{h}\wedge\bar{H}^{h}\right)\cdot\phi dxdt=0,
\end{align*}
where $\wedge$ represents the outer product of a two-dimensional vector.
The energy corresponding to the limiting magnetization field $\bar{\textbf{H}}_{0}$ of the effective field $\bar{\textbf{H}}^{h}$ is
$$\bar{E}_{0}(u)=\frac{\varepsilon}{2}\int_{\Omega}|\nabla'u|^{2}dx+\frac{1}{2}\int_{\mathbb{R}^{3}}|\nabla'v|^{2}dx,$$
where $\Delta v=div(u\chi_{\Omega})\otimes\delta_{\{x_{3}=0\}}.$
Because some stray field energy is converted into kinetic energy in the limit process, then the total energy of the limit field is
$$\bar{E}_{tot}(u)=\bar{E}_{0}(u)+\frac{\beta}{2}\int_{\Omega}|\omega|^{2}dx,$$
where $-\omega$ is the angular velocity, $\frac{1}{\beta}=\lim\limits_{h\rightarrow0}h[\gamma(h)]^{2}.$
Then consider the kinetic energy term of the limit field. We make a polar transformation of the vector $u$ in the plane. If $|u|=0$, then $u$ is the zero vector. And it's easy to know that the angular velocity is zero in this case and the kinetic energy is zero. So we consider the case of $|u|\neq0$. Let
$$\frac{u}{|u|}=(cos\theta,sin\theta),$$
one has
$$\partial_{t}\left(\frac{u}{|u|}\right)=\left(-sin\theta\partial_{t}\theta,cos\theta\partial_{t}\theta\right),$$
$$\frac{u}{|u|}\wedge\partial_{t}\left(\frac{u}{|u|}\right)=cos^{2}\theta\partial_{t}\theta+sin^{2}\theta\partial_{t}\theta=\partial_{t}\theta.$$
Thus the angular velocity is $-\omega=\frac{u}{|u|}\wedge\partial_{t}\left(\frac{u}{|u|}\right)$ and
$$\omega_{t}=-\partial_{t}\left(\frac{u}{|u|}\wedge\partial_{t}\left(\frac{u}{|u|}\right)\right)=-\frac{1}{|u|^{2}}u\wedge\partial_{t}^{2}u.$$
Therefore we have
\begin{align*}
\lim\limits_{h\rightarrow0}\int_{0}^{T}\int_{\Omega(h)}\left(u\wedge\bar{H}^{h}\right)\cdot\phi dxdt&=\int_{0}^{T}\int_{\Omega}\left[u\wedge\bar{H}_{0}+\beta\omega_{t}\right]\phi dxdt\\
&=\int_{0}^{T}\int_{\Omega}\left[u\wedge(\varepsilon\Delta'u-\nabla'v)-
\frac{\beta}{|u|^{2}}u\wedge\partial_{t}^{2}u\right]\phi dxdt=0,
\end{align*}
where $\Delta'$ and $\nabla'$ represent Laplace operators and gradients of two-dimensional vectors, respectively.
In order to avoid singularity, multiplying both sides of the equation by $|u|^{2}$, then the LLB limit equation formally reads
$$|u|^{2}u\wedge(\varepsilon\Delta'u-\nabla'v)-\beta u\wedge\partial_{t}^{2}u=0.$$

Under certain parameter mechanisms and assumptions, we will establish a three-dimensional model of LLB equation \eqref{1.10} on the film and study its initial boundary value problem. Thus, when the film thickness tends to zero, the limit equation is obtained and the limit process is proved in a mathematically rigorous way. The rest of the paper is arranged as follows. In Section \ref{section2}, we will prove the existence of weak solutions of equation \eqref{1.10} considering the action of stray field (see Theorem \ref{theorem 2.1}), and then obtain the necessary compactness from the energy estimates. We also study the interaction of static magnetic field (stray field) without considering the setting of time variable, and get some lemmas. Finally, we state the main results of this paper (see Theorem \ref{theorem4.1}) in Section \ref{section3}, and prove main theorem in Section \ref{section4}.

\section{Preliminaries}\label{section2}
Let's  first introduce some notations that will be used in the paper. We define the function spaces as follows:
$$W^{k,p}(\mathbb{R}^{n})(1\leq p\leq\infty)=\left\{f\in L^{p}\left(\mathbb{R}^{n}\right): \nabla^{\alpha}f\in L^{p}\left(\mathbb{R}^{n}\right),|\alpha|\leq k\right\},$$
$$H^{k}(\mathbb{R}^{n})=W^{k,2}(\mathbb{R}^{n}),$$
$$\dot{H}^{1}(\mathbb{R}^{n})=\{f\in L^{6}(\mathbb{R}^{n}):\nabla f\in L^{2}(\mathbb{R}^{n})\}.$$
Here, $L^{p}(\mathbb{R}^{n})(1\leq p\leq+\infty)$ stands for pth-power Lebesgue integrable function spaces. $H_{0}^{1}(\mathbb{R}^{n})$ is a set of functions belonging to $H^{1}(\mathbb{R}^{n})$ and having compact support on $\mathbb{R}^{n}$. $H^{-1}(\mathbb{R}^{n})$ is dual space of $H_{0}^{1}(\mathbb{R}^{n})$. In this paper, $\langle\cdot,\cdot\rangle$ denotes dual product, $\times$ represents the standard vector product on $\mathbb{R}^{3}$. And $\wedge$ represents the outer product on $\mathbb{R}^{2}$, for $p, q\in\mathbb{R}^{2},$
$$p\wedge q=p_{1}q_{2}-p_{2}q_{1}.$$
We use the standard symbols for the gradient $\nabla$ and the Laplacian $\Delta=\nabla\cdot\nabla$  acting on functions on $\mathbb{R}^{3}$. The corresponding planar operators are
$$\nabla'=\left(\frac{\partial}{\partial x_{1}},\frac{\partial}{\partial x_{2}}\right),\quad \Delta'=\nabla'\cdot\nabla'.$$
Symbol $\chi_{E}$ represents the characteristic function of measurable set $E\subset\mathbb{R}^{n}$, and the integral average is defined as
$$\fint_{E}f(y)dy=\frac{1}{|E|}\int_{E}f(y)dy.$$
Finally, we mark the vector in $\mathbb{R}^{3}$ with bold letters to distinguish it from the vector on the plane, such as $\textbf{X}=(X,X_{3})\in\mathbb{R}^{3}$, where $X\in\mathbb{R}^{2}$.

Now, we prove the existence of weak solutions of equation \eqref{1.10} with $\textbf{u}_{0}$ as initial value and the homogeneous Neumann boundary conditions on bounded open set $\mathcal{B}\subset\mathbb{R}^{3}$ with $C^{2}$ boundary.

\begin{definition}\label{definition 2.1}
Given $T>0$,  a weak solution $\textbf{u}:[0,T]\rightarrow H^{1}\cap L^{4}$ to \eqref{1.10} satisfies
\begin{align}\label{2.12}
<\textbf{u}(t),\phi>=&<\textbf{u}_{0},\phi>-LA\int_{0}^{t}<\nabla\textbf{u}(s),\nabla\phi>ds-L\int_{0}^{t}<\nabla U(s),\phi>ds\notag\\
&-\gamma A\int_{0}^{t}<\textbf{u}(s)\times\nabla\textbf{u}(s),\nabla\phi>ds-\gamma\int_{0}^{t}<\textbf{u}(s)\times\nabla U(s),\phi>ds\notag\\
&-\frac{L}{\chi_{11}}\int_{0}^{t}<(1+\mu|\textbf{u}|^{2}(s))\textbf{u}(s),\phi>ds,
\end{align}
for every $\phi\in C^{\infty}_{0}(\mathcal{B})$ and $t\in[0,T]$.
\end{definition}

The existence theorem of weak solution is given below.
\begin{theorem}\label{theorem 2.1}
Let $\mathcal{B}\subset\mathbb{R}^{3}$ be an open bounded domain with $C^{2}$ boundary, for any given $T>0$ and initial data $\textbf{u}(0)=\textbf{u}_{0}\in H^{1}$, there exists a weak solution of equation \eqref{1.10} such that $\textbf{u}\in L^{\infty}(0,T;H^{1})\cap L^{2}(0,T;H^{2})$.
\end{theorem}

\begin{proof}
We use the Faedo-Galerkin approximation method to prove Theorem \ref{theorem 2.1}. Compared to \cite{KNL}, we consider an equation with the addition of the stray field. The method is the same, here we only give a brief proof.

According to \cite{KNL}, we find that approximate solution $\textbf{u}_{n}\in S_{n}$ and $\Delta\textbf{u}_{n}\in S_{n}$. The induced stray field potential is $U_{n}$ and $\nabla U_{n}\in S_{n}$ by \eqref{1.3}. Then we get
\begin{align}\label{2.10}
\begin{cases}
\frac{\partial\textbf{u}_{n}}{\partial t}-LA\Delta\textbf{u}_{n}+L\nabla U_{n}-\gamma A\Pi_{n}(\textbf{u}_{n}\times\Delta\textbf{u}_{n})+\gamma\Pi_{n}(\textbf{u}_{n}\times\nabla U_{n})+\frac{L}{\chi_{11}}\Pi_{n}\left(\left(1+\mu|\textbf{u}_{n}|^{2}\right)\textbf{u}_{n}\right)=0,\\
\textbf{u}_{n}(x,0)=\textbf{u}_{0n},\quad\textbf{u}_{0n}\in S_{n},
\end{cases}
\end{align}
where $\textbf{u}_{0n}$ is the approximation of $\textbf{u}_{0}$.

In the following, we only consider the items related to the stray field, and see \cite{KNL} for details of the proof of other items.
For any $\textbf{u}_{1}$, $\textbf{u}_{2}\in S_{n}$, using the fact that the induced stray fields $\nabla U_{1}$, $\nabla U_{2}\in S_{n}$, and  \eqref{1.3}, we get
\begin{align*}
\|\nabla U_{1}-\nabla U_{2}\|_{L^{2}}&=\|\nabla(U_{1}-U_{2})\|_{L^{2}}\leq C\|\textbf{u}_{1}-\textbf{u}_{2}\|_{L^{2}},
\end{align*}
and
\begin{align*}
\|\Pi_{n}(\textbf{u}_{1}\times\nabla U_{1})\!-\!\Pi_{n}(\textbf{u}_{2}\times\nabla U_{2})\|_{L^{2}}&=\|\Pi_{n}(\textbf{u}_{1}\times\nabla U_{1}\!-\!\textbf{u}_{2}\times\nabla U_{2})\|_{L^{2}}\leq\|\textbf{u}_{1}\times\nabla U_{1}\!-\!\textbf{u}_{2}\times\nabla U_{2}\|_{L^{2}}\\
&\leq\|\textbf{u}_{1}\times(\nabla U_{1}-\nabla U_{2})\|_{L^{2}}+\|(\textbf{u}_{1}-\textbf{u}_{2})\times\nabla U_{2}\|_{L^{2}}\\
&\leq\|\textbf{u}_{1}\|_{L^{\infty}}\|\nabla U_{1}-\nabla U_{2}\|_{L^{2}}+\|\textbf{u}_{1}-\textbf{u}_{2}\|_{L^{2}}\|\nabla U_{2}\|_{L^{\infty}}\\
&\leq C\|\textbf{u}_{1}\|_{L^{\infty}}\|\textbf{u}_{1}-\textbf{u}_{2}\|_{L^{2}}+\|\textbf{u}_{1}-\textbf{u}_{2}\|_{L^{2}}\|\nabla U_{2}\|_{L^{\infty}}.
\end{align*}
By using the existence theorem of solutions of the ordinary differential equations, we can obtain the existence of approximate solutions. From \eqref{2.10},
we can easily establish the following estimates:
\begin{align}\label{2.20}
\|\textbf{u}_{n}(t)\|_{L^{2}}^{2}+2L\!\!\int_{0}^{T}\!\!\!\!A\|\nabla\textbf{u}_{n}(t)\|_{L^{2}}^{2}+\|\nabla U_{n}(t)\|_{L^{2}}^{2}dt
+\frac{2L}{\chi_{11}}\!\!\int_{0}^{T}\!\!\!\!\left(\|\textbf{u}_{n}(t)\|_{L^{2}}^{2}+\mu\|\textbf{u}_{n}(t)\|_{L^{4}}^{4}\right)dt\!
\leq\!\|\textbf{u}_{n}(0)\|_{L^{2}}^{2},
\end{align}
and
\begin{align}\label{2.21}
\|\nabla\textbf{u}_{n}(t)\|_{L^{2}}^{2}+2LA\int_{0}^{T}\|\Delta\textbf{u}_{n}(t)\|_{L^{2}}^{2}dt\leq\|\nabla\textbf{u}_{n}(0)\|_{L^{2}}^{2},
\end{align}
for $\forall n\in\mathbb{N}$, $t\in[0,T]$.

Finally, we consider the limit of the stray field term $\langle\nabla U_{n},\phi\rangle$  and $\langle\Pi_{n}(\textbf{u}_{n}\times\nabla U_{n}),\phi\rangle$.
From \cite{KNL}, there exist a subsequence of $\left\{\textbf{u}_{n}\right\}$ (still denoted by $\left\{\textbf{u}_{n}\right\}$) and $\textbf{u}\in L^{\infty}(0,T;H^{1})\cap L^{2}(0,T;H^{2})$ such that
$$\textbf{u}_{n}\rightharpoonup\textbf{u} \text{ in } L^{2}(0,T;H^{1}),$$
which together with \eqref{1.2} yields that
\begin{align}\label{2.5}
\lim\limits_{n\rightarrow\infty}\int_{0}^{T}\langle\nabla U_{n},\phi\rangle dt=\lim\limits_{n\rightarrow\infty}\int_{0}^{T}\langle\textbf{u}_{n},\phi\rangle dt=\int_{0}^{T}\langle\textbf{u},\phi\rangle dt=\int_{0}^{T}\langle\nabla U,\phi\rangle dt.
\end{align}

According to H\"{o}lder's inequality, \eqref{2.20}, and \eqref{2.21}, we have
$$\int_{0}^{T}\|\textbf{u}_{n}(t)\times\nabla U_{n}(t)\|_{L^{\frac{3}{2}}}^{2}dt\leq C\sup\limits_{t\in[0,T]}\|\textbf{u}_{n}\|_{H^{1}}^{2}\int_{0}^{T}\|\nabla U_{n}\|_{L^{2}}^{2}dt\leq C.$$
Then there exist a subsequence of $\{\textbf{u}_{n}\times\nabla U_{n}\}$ (still denoted by $\{\textbf{u}_{n}\times\nabla U_{n}\}$) and $Z\in L^{2}(0,T;L^{\frac{3}{2}})$, such that
$$\textbf{u}_{n}\times\nabla U_{n}\rightharpoonup Z \text{ in } L^{2}(0,T;L^{\frac{3}{2}}).$$
Further, we obtain
\begin{align*}
\int_{0}^{T}\|\Pi_{n}(\textbf{u}_{n}\times\nabla U_{n})\|_{X^{-\beta}}^{2}dt\leq\int_{0}^{T}\|\textbf{u}_{n}\times\nabla U_{n}\|^{2}_{X^{-\beta}}dt\leq C\int_{0}^{T}\|\textbf{u}_{n}\times\nabla U_{n}\|^{2}_{L^{\frac{3}{2}}}dt\leq C.
\end{align*}
There exist a subsequence of $\{\Pi_{n}(\textbf{u}_{n}\times\nabla U_{n})\}$ (still denoted by $\{\Pi_{n}(\textbf{u}_{n}\times\nabla U_{n})\}$) and $\bar{Z}\in L^{2}(0,T;X^{-\beta})$, such that
$$\Pi_{n}(\textbf{u}_{n}\times\nabla U_{n})\rightharpoonup \bar{Z} \text{ in } L^{2}(0,T;X^{-\beta}).$$
It follows from \cite[Lemma 4.2]{KNL} that $Z=\bar{Z}$ in $L^{2}(0,T;X^{-\beta})$.
Then for any $\phi\in L^{4}(0,T;L^{4})\cap L^{2}(0,T;X^{\beta})$, we have
\begin{align*}
\lim\limits_{n\rightarrow\infty}\int_{0}^{T}\langle\Pi_{n}(\textbf{u}_{n}\times\nabla U_{n}),\phi\rangle dt=\lim\limits_{n\rightarrow\infty}\int_{0}^{T}\langle\textbf{u}_{n}\times\nabla U_{n},\phi\rangle dt.
\end{align*}
Applying H\"{o}lder's inequality, we find that
\begin{align*}
&\quad\left|\int_{0}^{T}\langle\textbf{u}_{n}\times\nabla U_{n},\phi\rangle dt-\int_{0}^{T}\langle\textbf{u}\times\nabla U,\phi\rangle dt\right|\\
&\leq\left|\int_{0}^{T}\langle(\textbf{u}_{n}-\textbf{u})\times\nabla U_{n},\phi\rangle dt\right|+\left|\int_{0}^{T}\langle\textbf{u}\times(\nabla U_{n}-\nabla U),\phi\rangle dt\right|\\
&\leq\|\textbf{u}_{n}-\textbf{u}\|_{L^{4}(0,T;L^{4})}\|\nabla U_{n}\|_{L^{2}(0,T;L^{2})}\|\nabla\phi\|_{L^{4}(0,T;L^{4})}+\left|\int_{0}^{T}\langle\nabla U_{n}-\nabla U,\phi\times\textbf{u}\rangle dt\right|\\
&\leq C\|\textbf{u}_{n}-\textbf{u}\|_{L^{4}(0,T;L^{4})}+\left|\int_{0}^{T}\langle\nabla U_{n}-\nabla U,\phi\times\textbf{u}\rangle dt\right|.
\end{align*}
From the basic energy estimates, we can infer that $\textbf{u}_{n}\rightarrow\textbf{u}\text{ in } L^{4}(0,T;L^{4})$
and $\textbf{u}\in L^{4}(0,T;L^{4})$.
Thus, we obtain
$$\lim\limits_{n\rightarrow\infty}\int_{0}^{T}\langle\Pi_{n}(\textbf{u}_{n}\times\nabla U_{n}),\phi\rangle dt=\int_{0}^{T}\langle\textbf{u}\times\nabla U,\phi\rangle dt.$$
Combining \eqref{2.5} and the above arguments, we can get the desired result.
\end{proof}
Next, we establish the energy estimates of equation \eqref{1.10}. Let $\textbf{u}^{h}$ be a family of weak solutions of LLB equation \eqref{1.10} on $\Omega(h)\times(0,T)$, and $U^{h}$ is the corresponding stray field potential, then there are
\begin{align}\label{6.8}
\frac{\partial\textbf{u}^{h}}{\partial t}=\gamma A\textbf{u}^{h}\times\Delta\textbf{u}^{h}-\gamma\textbf{u}^{h}\times\nabla U^{h}+LA\Delta\textbf{u}^{h}-\frac{L}{\chi_{11}}\textbf{u}^{h}-\frac{L\mu}{\chi_{11}}\left|\textbf{u}^{h}\right|^{2}\textbf{u}^{h}-L\nabla U^{h},
\end{align}
and
$$\textbf{u}^{h}\in L^{\infty}(0,T;H^{1}(\Omega(h)))\cap L^{2}(0,T;H^{2}(\Omega(h))).$$
It follows from \eqref{1.3}  that
$$U^{h}\in L^{\infty}(0,T;\dot{H}^{1}(\mathbb{R}^{3})).$$
By using the above estimates and \eqref{6.8}, we can conclude that $\textbf{u}_{t}^{h}\in L^{2}(0,T;L^{\frac{3}{2}}(\Omega(h))).$

In order to facilitate the proof of the later limit equation, we define the renormalized stray field potential $v^{h}=\frac{U^{h}}{h}.$

Some lemmas about energy estimates are given below.
\begin{lemma}\label{6.10}
If $\sup\limits_{h}\frac{1}{h^{2}}\fint_{\Omega(h)}\left|\textbf{u}^{h}(0)\right|^{2}dx<\infty$, when $h\rightarrow0$, $\gamma(h)\sqrt{h}\rightarrow1$, $\frac{L(h)}{\gamma(h)h}\rightarrow a$, then
\begin{align*}
\sup\limits_{h,t}\fint_{\Omega(h)}\left|\textbf{u}^{h}\right|^{2}dx<\infty,\; \sup\limits_{h}\int_{0}^{T}\!\!\!\fint_{\Omega(h)}\left|\textbf{u}^{h}\right|^{4}dxdt<\infty, \; \sup\limits_{h,t}\frac{1}{h}\int_{\mathbb{R}^{3}}\left|\nabla v^{h}\right|^{2}dxdt<\infty.
\end{align*}
\end{lemma}

\begin{proof}
Taking the limit on both sides of \eqref{2.20} and using the weak lower semicontinuity of norm, we obtain
\begin{align}\label{6.1}
&\fint_{\Omega(h)}\left|\textbf{u}^{h}\right|^{2}dx+2LA\int_{0}^{T}\!\!\!\fint_{\Omega(h)}
\left|\nabla\textbf{u}^{h}\right|^{2}dxdt+2\frac{L}{\chi_{11}}\int_{0}^{T}\!\!\!\fint_{\Omega(h)}\left|\textbf{u}^{h}\right|^{2}dxdt+2\frac{L\mu}{\chi_{11}}\int_{0}^{T}\!\!\!
\fint_{\Omega(h)}\left|\textbf{u}^{h}\right|^{4}dxdt\notag\\
&+2Lh\int_{0}^{T}\!\!\!\int_{\mathbb{R}^{3}}\left|\nabla v^{h}\right|^{2}dxdt\leq\fint_{\Omega(h)}\left|\textbf{u}^{h}(0)\right|^{2}dx.
\end{align}
Then we have
$$\sup\limits_{h,t}\fint_{\Omega(h)}\left|\textbf{u}^{h}\right|^{2}dx\leq\sup\limits_{h}\fint_{\Omega(h)}\left|\textbf{u}^{h}(0)\right|^{2}dx
\leq\sup\limits_{h}\frac{1}{h^{2}}\fint_{\Omega(h)}\left|\textbf{u}^{h}(0)\right|^{2}dx<\infty,$$
and
$$2\frac{L}{\sqrt{h}}\frac{\mu}{\chi_{11}}\int_{0}^{T}\!\!\!\fint_{\Omega(h)}\left|\textbf{u}^{h}\right|^{4}dxdt\leq\frac{1}{\sqrt{h}}
\fint_{\Omega(h)}\left|\textbf{u}^{h}(0)\right|^{2}dx\leq\frac{1}{h^{2}}\fint_{\Omega(h)}\left|\textbf{u}^{h}(0)\right|^{2}dx<\infty.$$
i.e.,
$$\sup\limits_{h}\int_{0}^{T}\!\!\!\fint_{\Omega(h)}\left|\textbf{u}^{h}\right|^{4}dxdt\lesssim\sup\limits_{h}\frac{1}{h^{2}}\fint_{\Omega(h)}\left
|\textbf{u}^{h}(0)\right|^{2}dx<\infty.$$
We obtain from \eqref{1.3} that
$$\sup\limits_{h,t}\frac{1}{h}\int_{\mathbb{R}^{3}}\left|\nabla v^{h}\right|^{2}dx\lesssim\sup\limits_{h,t}\frac{1}{h^{2}}\fint_{\Omega(h)}\left|\textbf{u}^{h}\right|^{2}dx,$$
which combined with \eqref{6.1} yields that
$$\sup\limits_{h,t}\frac{1}{h}\int_{\mathbb{R}^{3}}\left|\nabla v^{h}\right|^{2}dxdt\lesssim\sup\limits_{h}\frac{1}{h^{2}}\fint_{\Omega(h)}\left|\textbf{u}^{h}(0)\right|^{2}dx<\infty.$$
\end{proof}

\begin{lemma}\label{6.6}
If $\sup\limits_{h}\frac{1}{h\sqrt{h}}\fint_{\Omega(h)}\left|\nabla\textbf{u}^{h}(0)\right|^{2}dx<\infty$, when $h\rightarrow0$, $\frac{A(h)}{h}\rightarrow\varepsilon$, $\gamma(h)\sqrt{h}\rightarrow1$, $\frac{L(h)}{\gamma(h)h}\rightarrow a$, then
\begin{align*}
\sup\limits_{h,t}\frac{1}{h\sqrt{h}}\fint_{\Omega(h)}\left|\nabla\textbf{u}^{h}\right|^{2}dx<\infty,\; \sup\limits_{h}\int_{0}^{T}\fint_{\Omega(h)}\left|\Delta\textbf{u}^{h}\right|^{2}dxdt<\infty,\; \sup\limits_{h,t}\fint_{\Omega(h)}\left|\textbf{u}^{h}\right|^{6}dx<\infty,
\end{align*}
\begin{align*}
\sup\limits_{h}\int_{0}^{T}\left(\fint_{\Omega(h)}\left|\nabla\textbf{u}^{h}\right|^{6}dx\right)^{\frac{1}{3}}dt<\infty,\; \sup\limits_{h}\int_{0}^{T}\left(\fint_{\Omega(h)}\left|\nabla\textbf{u}^{h}\right|^{3}dx\right)^{\frac{2}{3}}dt<\infty.
\end{align*}
\end{lemma}

\begin{proof}
Taking the limit on both sides of \eqref{2.21}, we get from the weak lower semicontinuity of norm that
\begin{align*}
&\fint_{\Omega(h)}\left|\nabla\textbf{u}^{h}\right|^{2}dx+2LA\int_{0}^{T}\fint_{\Omega(h)}\left|\Delta\textbf{u}^{h}\right|^{2}dxdt
\leq\fint_{\Omega(h)}\left|\nabla\textbf{u}^{h}(0)\right|^{2}dx.
\end{align*}
Then we have
\begin{align}\label{6.30}
\frac{1}{h\sqrt{h}}\sup\limits_{h,t}\fint_{\Omega(h)}\left|\nabla\textbf{u}^{h}\right|^{2}dx\leq\sup\limits_{h}
\frac{1}{h\sqrt{h}}\fint_{\Omega(h)}\left|\nabla\textbf{u}^{h}(0)\right|^{2}dx<\infty,
\end{align}
and
$$2\frac{L}{\sqrt{h}}\frac{A}{h}\int_{0}^{T}\fint_{\Omega(h)}\left|\Delta\textbf{u}^{h}\right|^{2}dxdt\leq\frac{1}{h\sqrt{h}}\fint_{\Omega(h)}
\left|\nabla\textbf{u}^{h}(0)\right|^{2}dx,$$
i.e.,
\begin{equation}\label{6.15}
\sup\limits_{h}\int_{0}^{T}\fint_{\Omega(h)}\left|\Delta\textbf{u}^{h}\right|^{2}dxdt\lesssim\sup\limits_{h}\frac{1}{h\sqrt{h}}\fint_{\Omega(h)}
\left|\nabla\textbf{u}^{h}(0)\right|^{2}dx<\infty.
\end{equation}
From the Gagliardo-Nirenberg (G-N) inequality, we obtain
$$\left\|\textbf{u}^{h}\right\|_{L^{6}}\lesssim\left\|\nabla\textbf{u}^{h}\right\|_{L^{2}},$$
which together with \eqref{6.30} implies  that
\begin{align}\label{6.50}
\sup\limits_{h,t}\fint_{\Omega(h)}\left|\textbf{u}^{h}\right|^{6}dx&\lesssim\frac{1}{h}\left(\sup\limits_{h,t}\int_{\Omega(h)}\left|\nabla\textbf{u}^{h}\right|^{2}dx\right)
^{3}=h^{2}\left(\sup\limits_{h,t}\fint_{\Omega(h)}\left|\nabla\textbf{u}^{h}\right|^{2}dx\right)^{3}<\infty.
\end{align}
Utilizing the G-N inequality again, we obtain
$$\left\|\nabla\textbf{u}^{h}\right\|_{L^{6}}\lesssim\left\|\Delta\textbf{u}^{h}\right\|_{L^{2}}.$$
It holds that
\begin{align*}
\int_{0}^{T}\left(\fint_{\Omega(h)}\left|\nabla\textbf{u}^{h}\right|^{6}dx\right)^{\frac{1}{3}}dt\lesssim h^{\frac{2}{3}}\int_{0}^{T}\fint_{\Omega(h)}\left|\Delta\textbf{u}^{h}\right|^{2}dxdt.
\end{align*}
Combining this above inequality and \eqref{6.15}, we have
$$\sup\limits_{h}\int_{0}^{T}\left(\fint_{\Omega(h)}\left|\nabla\textbf{u}^{h}\right|^{6}dx\right)^{\frac{1}{3}}dt\lesssim\sup\limits_{h}\int_{0}^{T}
\fint_{\Omega(h)}\left|\Delta\textbf{u}^{h}\right|^{2}dxdt<\infty.$$
By employing  H\"{o}lder's inequality, we also have
\begin{align*}
\int_{0}^{T}\left(\fint_{\Omega(h)}\left|\nabla\textbf{u}^{h}\right|^{3}dx\right)^{\frac{2}{3}}dt\lesssim\int_{0}^{T}\left(\fint_{\Omega(h)}\left|\nabla\textbf{u}^{h}\right|^{6}dx\right)^{\frac{1}{3}}dt,
\end{align*}
then
$$\sup\limits_{h}\int_{0}^{T}\left(\fint_{\Omega(h)}\left|\nabla\textbf{u}^{h}\right|^{3}dx\right)^{\frac{2}{3}}dt\lesssim\sup\limits_{h}\int_{0}^{T}\left(\fint_{\Omega(h)}
\left|\nabla\textbf{u}^{h}\right|^{6}dx\right)^{\frac{1}{3}}dt<\infty$$
holds.
\end{proof}

\begin{lemma}\label{6.7}
If $\sup\limits_{h}\frac{1}{h^{2}}\fint_{\Omega(h)}\left|\textbf{u}^{h}(0)\right|^{2}dx<\infty$ and $\sup\limits_{h}\frac{1}{h\sqrt{h}}\fint_{\Omega(h)}\left|\nabla\textbf{u}^{h}(0)\right|^{2}dx<\infty$, when $h\rightarrow0$, $\frac{A(h)}{h}\rightarrow\varepsilon$, $\gamma(h)\sqrt{h}\rightarrow1$, $\frac{L(h)}{\gamma(h)h}\rightarrow a$, then it holds that
$$\sup\limits_{h}\int_{0}^{T}\left(\fint_{\Omega(h)}\left|\textbf{u}_{t}^{h}\right|^\frac{3}{2}dx\right)^{\frac{4}{3}}dt<\infty.$$
\end{lemma}

\begin{proof}
Since $\sup\limits_{h}\int_{0}^{T}\left(\fint_{\Omega(h)}\left|\textbf{u}_{t}^{h}\right|^\frac{3}{2}dx\right)^{\frac{4}{3}}dt=\sup\limits_{h}\left\|\frac{1}{h^{\frac{2}{3}}}
\textbf{u}_{t}^{h}\right\|^{2}_{L^{2}(0,T;L^{\frac{3}{2}})}$, we divide both sides of equation \eqref{6.8} by $h^{\frac{2}{3}}$ and then obtain
$$\frac{1}{h^{\frac{2}{3}}}\textbf{u}^{h}_{t}=\frac{1}{h^{\frac{2}{3}}}\left(\gamma A\textbf{u}^{h}\times\Delta\textbf{u}^{h}-\gamma\textbf{u}^{h}\times\nabla U^{h}+LA\Delta\textbf{u}^{h}-\frac{L}{\chi_{11}}\textbf{u}^{h}-\frac{L\mu}{\chi_{11}}\left|\textbf{u}^{h}\right|^{2}\textbf{u}^{h}-L\nabla U^{h}\right).$$
Using the G-N inequality and H\"{o}lder's inequality, it is easy to verify that each term on the right-hand side of the above equation belongs to $L^{2}(0,T;L^{\frac{3}{2}})$ for any $h>0$.
\end{proof}

To study the interaction of stray field on thin magnet in simple environment, only electrostatic magnetic field needs to be considered. At the end of this section, we give some lemmas of stray field interaction without considering time variables.

Considering electrostatic magnetic field
$$\textbf{u}^{h}:~\Omega(h)\rightarrow\mathbb{R}^{3},\quad\textbf{u}^{h}\in H^{1}(\Omega(h)),$$
by H\"{o}lder's inequality and the G-N inequality, one has $\textbf{u}^{h}\in L^{p}(\Omega(h))$ for $1\leq p\leq 6$.

From the potential equation $\Delta U^{h}=div(\textbf{u}^{h}\chi_{\Omega(h)})$, its unique solution $U^{h}\in \dot{H}^{1}(\mathbb{R}^{3})$ satisfies
\begin{align}\label{3.2}
\int_{\mathbb{R}^{3}}\nabla U^{h}\cdot\nabla\phi dx=\int_{\Omega(h)}\textbf{u}^{h}\cdot\nabla\phi dx,\quad\phi\in \dot{H}^{1}(\mathbb{R}^{3}).
\end{align}
And according to \eqref{1.3}, we have
$$\left\|\nabla U^{h}\right\|_{L^{6}(\mathbb{R}^{3})}\lesssim\left\|\textbf{u}^{h}\right\|_{L^{6}(\Omega(h))}<\infty,$$
which yields that the weak form of potential equation can be also written as
\begin{align}\label{3.3}
\int_{\mathbb{R}^{3}}\nabla U^{h}\cdot\nabla\phi dx=\int_{\Omega(h)}\textbf{u}^{h}\cdot\nabla\phi dx,\quad\phi\in W^{1,\frac{6}{5}}(\mathbb{R}^{3}).
\end{align}
Noticing that $\nabla\phi\in L^{\frac{6}{5}}(\mathbb{R}^{3})$, and we can naturally obtain $\phi\in L^{2}(\mathbb{R}^{3})$ by the G-N inequality.

In order to determine the asymptotic limit, two important physical quantities in the vertical direction are defined by
$$w_{1}^{h}=\frac{1}{\sqrt{h}}\fint_{0}^{h}\left|\textbf{u}^{h}\right|^{2}u_{3}^{h}dx_{3},\quad w_{2}^{h}=\frac{1}{\sqrt{h}}\fint_{0}^{h}\left|\textbf{u}^{h}\right|^{2}\frac{\partial U^{h}}{\partial x_{3}}dx_{3}.$$
Renormalized stray field potential is given by $v^{h}=\frac{U^{h}}{h}$. The following lemma implies the $L^{\frac{6}{5}}$ estimates of $w_{1}^{h}$ and $w_{2}^{h}$.
\begin{lemma}\label{lemma3.1}
If $\sup\limits_{h}\fint_{\Omega(h)}\left|\nabla\textbf{u}^{h}\right|^{2}dx<\infty$ and $\sup\limits_{h}\int_{\mathbb{R}^{3}}\left|\nabla v^{h}\right|^{2}dxdt<\infty$ hold, then $w_{1}^{h}\in L^{\frac{6}{5}}(\Omega),\,w_{2}^{h}\in L^{\frac{6}{5}}(\Omega).$
\end{lemma}
\begin{proof}
Using H\"{o}lder's inequality and \eqref{6.50}, we conclude that
\begin{align*}
\int_{\Omega}\left|w_{2}^{h}\right|^{\frac{6}{5}}dx&=\int_{\Omega}\frac{1}{h^{\frac{3}{5}}}\left|\fint_{0}^{h}\left|\textbf{u}^{h}\right|^{2}\frac{\partial U^{h}}{\partial x_{3}}dx_{3}\right|^{\frac{6}{5}}dx\leq\left(\fint_{\Omega(h)}\left|\textbf{u}^{h}\right|^{6}dx\right)^{\frac{2}{5}}\left(\int_{\mathbb{R}^{3}}\left|\frac{\partial v^{h}}{\partial x_{3}}\right|^{2}dx\right)^{\frac{3}{5}}<\infty,\quad\forall h.
\end{align*}
According to \cite[Lemma 2.1]{CM}, we have
$$\int_{\Omega(h)}\left|u_{3}^{h}\right|^{2}dx\lesssim\int_{\mathbb{R}^{3}}\left|\nabla U^{h}\right|^{2}dx+h\int_{\Omega(h)}\left|\nabla u_{3}^{h}\right|^{2}dx+h^{2}\left|\Omega\right|.$$
By H\"{o}lder's inequality and \eqref{6.50}, we obtain
\begin{align*}
\int_{\Omega}\left|w_{1}^{h}\right|^{\frac{6}{5}}dx&=\int_{\Omega}\frac{1}{h^{\frac{3}{5}}}\left|\fint_{0}^{h}\left|\textbf{u}^{h}\right|^{2}u_{3}^{h}
dx_{3}\right|^{\frac{6}{5}}dx
\leq\frac{1}{h^{\frac{8}{5}}}\left(\int_{\Omega(h)}\left|\textbf{u}^{h}\right|^{6}dx\right)^{\frac{2}{5}}\left(\int_{\Omega(h)}\left|u^{h}_{3}\right|
^{2}dx\right)^{\frac{3}{5}}\\
&\lesssim\left(\fint_{\Omega(h)}\left|\textbf{u}^{h}\right|^{6}dx\right)^{\frac{2}{5}}\left(\int_{\mathbb{R}^{3}}\left|\nabla v^{h}\right|^{2}dx+\fint_{\Omega(h)}\left|\nabla u_{3}^{h}\right|^{2}dx+\left|\Omega\right|\right)^{\frac{3}{5}}<\infty,\quad\forall h.
\end{align*}
\end{proof}

Making some changes to the methods used in Lemma 2.2 and Corollary 2.1 of \cite{CM}, we can get the following lemma.
\begin{lemma}\label{lemma3.2}
If $\sup\limits_{h}\fint_{\Omega(h)}\left|\nabla\textbf{u}^{h}\right|^{2}dx<\infty$ and $\sup\limits_{h}\int_{\mathbb{R}^{3}}\left|\nabla v^{h}\right|^{2}dxdt<\infty$ hold, then $w_{1}^{h}$ and $w_{2}^{h}$ has the same weak limit in $L^{\frac{6}{5}}(\Omega)$.
\end{lemma}
\begin{proof}
Extending $\varphi\in C_{0}^{\infty}(\Omega)$ to $\mathbb{R}^{3}$, we have
$$\phi_{h}(x,x_{3})=h\left|\textbf{u}^{h}\right|^{2}\psi\left(\frac{x_{3}}{h}\right)\rho(x_{3})\varphi(x),$$
where $\psi(x_{3})=\int_{0}^{x_{3}}\chi_{(0,1)}(z)dz$, $\rho\in C_{0}^{\infty}((-1,1))$ is a positive cut-off function. When $z\in(0,\frac{1}{2})$, one has $\rho(z)=1$ and $|\rho'(z)|\leq1$ for $\forall z\in\mathbb{R}$.

Through simple calculation, we can get $\phi_{h}\in W^{1,\frac{6}{5}}(\mathbb{R}^{3})$, and for $h\in (0,\frac{1}{2})$,
$$\nabla\phi_{h}=2h\textbf{u}^{h}\cdot\nabla\textbf{u}^{h}\psi\left(\frac{x_{3}}{h}\right)\rho(x_{3})\varphi(x)+\left|\textbf{u}^{h}\right|^{2}\varphi(x)\chi_{(0,h)}
(x_{3})\hat{e}_{3}+h\left|\textbf{u}^{h}\right|^{2}\psi\left(\frac{x_{3}}{h}\right)\nabla(\varphi\rho)(x).$$
Substituting $\phi=\phi_{h}$ into \eqref{3.3}, we obtain
\begin{align*}
&\left|\int_{\Omega(h)}\left|\textbf{u}^{h}\right|^{2}\left(u_{3}^{h}-\frac{\partial U^{h}}{\partial x_{3}}\right)\varphi(x)dx\right|\\
&=\left|-h\int_{\Omega(h)}2\left(\textbf{u}^{h}-\nabla U^{h}\right)\textbf{u}^{h}\cdot\nabla\textbf{u}^{h}\psi\left(\frac{x_{3}}{h}\right)\rho(x_{3})\varphi(x)+\left(\textbf{u}^{h}-\nabla U^{h}\right)\left|\textbf{u}^{h}\right|^{2}\psi\left(\frac{x_{3}}{h}\right)\nabla(\varphi\rho)(x)dx\right|\\
&\lesssim h\int_{\Omega(h)}\left|\textbf{u}^{h}\right|^{2}\left|\nabla\textbf{u}^{h}\right|dx+h\int_{\Omega(h)}\left|\nabla U^{h}\right|\cdot\left|\textbf{u}^{h}\right|\cdot\left|\nabla\textbf{u}^{h}\right|dx+h\int_{\Omega(h)}\left|\textbf{u}^{h}\right|^{3}+\left|\nabla U^{h}\right|\cdot\left|\textbf{u}^{h}\right|^{2}dx,
\end{align*}
From \eqref{1.3}, we have
$$\left\|\nabla U^{h}\right\|_{L^{4}(\mathbb{R}^{3})}\lesssim\left\|\textbf{u}^{h}\right\|_{L^{4}(\Omega(h))}.$$
We utilize H\"{o}lder's inequality to find that
\begin{align*}
\left|\int_{\Omega}\left(w_{1}^{h}-w_{2}^{h}\right)\varphi(x)dx\right|&\lesssim\frac{1}{\sqrt{h}}\int_{\Omega(h)}\left|\textbf{u}^{h}\right|^{2}\left|\nabla\textbf{u}^{h}\right|dx
+\frac{1}{\sqrt{h}}\int_{\Omega(h)}\left|\nabla U^{h}\right|\cdot\left|\textbf{u}^{h}\right|\cdot\left|\nabla\textbf{u}^{h}\right|dx\\
&\quad+\frac{1}{\sqrt{h}}\int_{\Omega(h)}\left|\textbf{u}^{h}\right|^{3}dx+\frac{1}{\sqrt{h}}\int_{\Omega(h)}\left|\nabla U^{h}\right|\cdot\left|\textbf{u}^{h}\right|^{2}dx\\
&\lesssim\sqrt{h}\left(\fint_{\Omega(h)}\left|\textbf{u}^{h}\right|^{4}dx\right)^{\frac{1}{2}}\left(\fint_{\Omega(h)}\left|\nabla\textbf{u}^{h}\right|^{2}dx\right)^{\frac{1}{2}}\\
&\quad+h^{\frac{1}{4}}\left(\int_{\Omega(h)}\left|\textbf{u}^{h}\right|^{4}dx\right)^{\frac{1}{4}}\left(\fint_{\Omega(h)}\left|\textbf{u}^{h}\right|^{4}dx\right)^{\frac{1}{4}}
\left(\fint_{\Omega(h)}\left|\nabla\textbf{u}^{h}\right|^{2}dx\right)^{\frac{1}{2}}\\
&\quad+\sqrt{h}\fint_{\Omega(h)}\left|\textbf{u}^{h}\right|^{3}dx+h\left(\int_{\mathbb{R}^{3}}\left|\nabla v^{h}\right|^{2}dx\right)^{\frac{1}{2}}\left(\fint_{\Omega(h)}\left|\textbf{u}^{h}\right|^{4}dx\right)^{\frac{1}{2}}.
\end{align*}
Using H\"{o}lder's inequality again, one has
$$\fint_{\Omega(h)}\left|\textbf{u}^{h}\right|^{4}dx\lesssim\left(\fint_{\Omega(h)}\left|\textbf{u}^{h}\right|^{6}dx\right)^{\frac{2}{3}},$$
$$\fint_{\Omega(h)}\left|\textbf{u}^{h}\right|^{3}dx\lesssim\left(\fint_{\Omega(h)}\left|\textbf{u}^{h}\right|^{6}dx\right)^{\frac{1}{2}},$$
which together with \eqref{6.50} and lemma \ref{lemma3.1} imply that when $h\rightarrow0$,
$$w_{1}^{h}\rightharpoonup w_{2}^{h} \text{ in } L^{\frac{6}{5}}(\Omega).$$
Then $w_{1}^{h}$ and $w^{h}_{2}$ have the same weak limit in $L^{\frac{6}{5}}(\Omega)$.
\end{proof}

In the following, the inequality of Lemma 2.4 in \cite{CM} about $L^{2}-L^{2}$ is extended to the case of $L^{p}-L^{q}$.
\begin{lemma}\label{lemma3.3}
If $f\in L^{p}(\Omega(h))$, $g\in L^{q}(\Omega(h))$ and $\frac{\partial f}{\partial x_{3}}\in L^{p}(\Omega(h))$, then
$$\left|\int_{\Omega}\left(\fint_{0}^{h}(f\cdot g)-\fint_{0}^{h}f\fint_{0}^{h}g\right)dx\right|\leq h\left(\fint_{\Omega(h)}\left|\frac{\partial f}{\partial x_{3}}\right|^{p}dx\right)^{\frac{1}{p}}\left(\fint_{\Omega(h)}|g|^{q}dx\right)^{\frac{1}{q}},$$
where $\frac{1}{p}+\frac{1}{q}=1$, $1<p$, $q<\infty.$
\end{lemma}
\begin{proof}
According to Lemma 2.4 in \cite{CM}, we have
$$\fint_{0}^{h}(f\cdot g)dx_{3}-\fint_{0}^{h}fdx_{3}\cdot\fint_{0}^{h}gdx_{3}=\fint_{0}^{h}\fint_{0}^{h}\left(\int_{z_{2}}^{z_{1}}\frac{\partial f}{\partial x_{3}}(x,y)dy\right)g(x,z_{1})dz_{1}dz_{2}.$$
Using H\"{o}lder's inequality, we calculate
\begin{align*}
&\quad\left|\int_{\Omega}\left(\fint_{0}^{h}\!(f\cdot g)dx_{3}-\fint_{0}^{h}\!fdx_{3}\cdot\fint_{0}^{h}\!gdx_{3}\right)dx\right|\!=\!\left|\int_{\Omega}\fint_{0}^{h}\!\!\fint_{0}^{h}\!\!\left(\int_{z_{2}}^{z_{1}}\!\!\frac{\partial f}{\partial x_{3}}(x,y)dy\right)g(x,z_{1})dz_{1}dz_{2}dx\right|\\
&=\left|\int_{\Omega}\fint_{0}^{h}\fint_{0}^{h}\left(\int_{0}^{h}\frac{\partial f}{\partial x_{3}}(x,y)dz_{2}\right)g(x,z_{1})dz_{1}dydx\right|\leq h\int_{\Omega}\fint_{0}^{h}\fint_{0}^{h}\left|\frac{\partial f}{\partial x_{3}}(x,y)\right|\left|g(x,z_{1})\right|dydz_{1}dx\\
&=h\int_{\Omega}\left(\fint_{0}^{h}\left|\frac{\partial f}{\partial x_{3}}(x,y)\right|dy\right)\left(\fint_{0}^{h}\left|g(x,z_{1})\right|dz_{1}\right)dx
\leq h\left(\fint_{\Omega(h)}\left|\frac{\partial f}{\partial x_{3}}\right|^{p}dx\right)^{\frac{1}{p}}\left(\fint_{\Omega(h)}\left|g\right|^{q}dx\right)^{\frac{1}{q}}.
\end{align*}
\end{proof}

Similar to the relationship between the magnetization in the limit field and its induced stray field proved by Lemma 2.3 in \cite{CM}, we can also obtain the relationship between spin polarization $\textbf{u}^{h}$ and its induced stray field potential.

\begin{lemma}\label{lemma3.4}
If
$$v^{h}\rightharpoonup v\text{ in } \dot{H}^{1}(\mathbb{R}^{3}),\;\,\fint_{0}^{h}\textbf{u}^{h}dx_{3}\rightharpoonup(u,0) \text{ in } L^{2}(\Omega),$$
while $h\int_{\Omega(h)}\left|\frac{\partial\textbf{u}^{h}}{\partial x_{3}}\right|^{2}dx\rightarrow0$, then $v$ is a weak solution of the following equation:
$$\int_{\mathbb{R}^{3}}\nabla v\cdot\nabla\varphi dx=\int_{\Omega}u(x)\nabla\varphi(x,0)dx,\ \ \forall\varphi\in C^{\infty}_{0}(\mathbb{R}^{3}).$$
\end{lemma}

Finally, we prove the regularity of the time derivative of the stray field induced by the magnetization field $\textbf{u}^{h}=\textbf{u}^{h}(t)$.
\begin{lemma}\label{lemma3.5}
If $\sup\limits_{h}\int_{0}^{\infty}\left(\fint_{\Omega(h)}\left|\textbf{u}_{t}^{h}\right|^{\frac{3}{2}}dx\right)^{\frac{4}{3}}dt<\infty$, then $v_{t}^{h}$ is uniformly bounded about $h$ in $L^{2}(\mathbb{R}^{3}\times(0,\infty))$.
\end{lemma}
\begin{proof}
Let $\psi\in C_{0}^{\infty}(\mathbb{R}^{3}\times(0,\infty))$ and $|\psi|\geq 1$. Then the associated Newton potential $\varphi\in C^{\infty}(\mathbb{R}^{3}\times(0,\infty))$ solving $\Delta\varphi(\cdot,t)=\psi(\cdot,t)$ is compactly supported in time (see \cite{LCE}).
Since $\Delta\varphi_{t}=\psi_{t}$, we have
$$\varphi_{t}=\Delta^{-1}\psi_{t},\quad\nabla\varphi_{t}=\nabla^{-1}\psi_{t},$$
and one has $\varphi_{t}\in\dot{H}^{1}(\mathbb{R}^{3})$ by the Hardy-Littlwood-Sobolev inequality. Similarly, $\varphi\in\dot{H}^{1}(\mathbb{R}^{3})$. For $v^{h}\in L^{\infty}((0,\infty);\dot{H}^{1}(\mathbb{R}^{3}))$,
we get by \eqref{3.2} that
\begin{align*}
\int_{0}^{\infty}\left\langle v^{h},\psi_{t}\right\rangle dt
&=\int_{0}^{\infty}\left\langle v^{h},\Delta\varphi_{t}\right\rangle dt=-\int_{0}^{\infty}\left\langle\nabla v^{h},\nabla\varphi_{t}\right\rangle dt=-\frac{1}{h}\int_{0}^{\infty}\left\langle\nabla U^{h},\nabla\varphi_{t}\right\rangle dt\\
&=-\frac{1}{h}\int_{0}^{\infty}\left\langle\textbf{u}^{h},\nabla\varphi_{t}\right\rangle dt=\frac{1}{h}\int_{0}^{\infty}\!\!\left\langle\textbf{u}^{h}_{t},\nabla\varphi\right\rangle dt,
\end{align*}
then we use H\"{o}lder's inequality to obtain
$$\left|\int_{0}^{\infty}\left\langle v^{h},\psi_{t}\right\rangle dt\right|\leq\!\left[\int_{0}^{\infty}\!\!\left(\fint_{\Omega(h)}\left|\textbf{u}^{h}_{t}\right|^{\frac{3}{2}}dx\right)^{\frac{4}{3}}\!dt\!\right]^{\frac{1}{2}}
\left[\int_{0}^{\infty}\!\!\left(\fint_{\Omega(h)}\left|\nabla\varphi\right|^{3}dx\right)^{\frac{2}{3}}\!dt\!\right]^{\frac{1}{2}}.$$
And one has
\begin{align*}
\fint_{0}^{h}\int_{\Omega}|\nabla\varphi|^{3}dxdx_{3}&\leq\sup\limits_{x_{3}\in[0,1]}\int_{\Omega}|\nabla\varphi|^{3}dx\leq\int_{\Omega(1)}\left|\nabla^{2}\varphi\right|
^{2}+\left|\nabla\varphi\right|^{3}dx\\
&\leq\left\|\nabla^{2}\varphi\right\|^{2}_{L^{2}(\Omega(1))}+C\left\|\nabla\varphi\right\|^{3}_{L^{6}(\Omega(1))}.
\end{align*}
Given an open set $U\supset\Omega(1)$, then $\varphi\in\dot{H}^{1}(U)$ by the fact that $\varphi\in\dot{H}^{1}(\mathbb{R}^{3})$. And we further obtain $\varphi\in H^{1}(U)$ from H\"{o}lder's inequality.

Since $\psi\in L^{2}(U)$, $\varphi\in H^{1}(U)$ is the solution of $\Delta\varphi=\psi$, then $\varphi\in H^{2}(\Omega(1))$ and by using the regularity of $H^{2}$ of elliptic equation in \cite{LCE}, we conclude that
$$\|\varphi\|_{H^{2}(\Omega(1))}\leq C(\|\psi\|_{L^{2}(U)}+\|\varphi\|_{L^{2}(U)})\leq C\|\psi\|_{L^{2}(U)}.$$
We further obtain $\nabla\varphi\in H^{1}(\Omega(1))$. One has $\nabla\varphi\in \dot{H}^{1}(\Omega(1))$ and
$$\left\|\nabla^{2}\varphi\right\|_{L^{2}(\Omega(1))}+\|\nabla\varphi\|_{L^{6}(\Omega(1))}\leq C\|\psi\|_{L^{2}(U)}.$$
Hence, we obtain
$$\fint_{\Omega(h)}|\nabla\varphi|^{3}dx\leq C\left(\|\psi\|^{2}_{L^{2}(U)}+\|\psi\|^{3}_{L^{2}(U)}\right)\leq C\|\psi\|^{3}_{L^{2}(U)}.$$
For $\forall h$, we have
$$\left|\int_{0}^{\infty}\langle v^{h},\psi_{t}\rangle dt\right|\leq C\left[\int_{0}^{\infty}\left(\|\psi\|^{3}_{L^{2}(U)}\right)^{\frac{2}{3}}dt\right]^{\frac{1}{2}}\leq C\left[\int_{0}^{\infty}\|\psi\|^{2}_{L^{2}(U)}dt\right]^{\frac{1}{2}},$$
and
$$\int_{0}^{\infty}\langle v^{h}_{t},\psi\rangle dt\leq\left|\int_{0}^{\infty}\langle v^{h}_{t},\psi\rangle dt\right|=\left|\int_{0}^{\infty}\langle v^{h},\psi_{t}\rangle dt\right|\leq C\left[\int_{0}^{\infty}\|\psi\|^{2}_{L^{2}(U)}dt\right]^{\frac{1}{2}}.$$
As for $\psi\in C_{0}^{\infty}(\mathbb{R}^{3}\times(0,\infty))$ and $C_{0}^{\infty}(\mathbb{R}^{3}\times(0,\infty))$ is dense in $L^{2}(\mathbb{R}^{3}\times(0,\infty))$. We conclude that $v_{t}^{h}$ is the bounded linear functional in $L^{2}(\mathbb{R}^{3}\times(0,\infty))$ for $\forall h$. So the lemma is proved.
\end{proof}

\section{The main result}\label{section3}
We suppose that $\Omega(h)=\Omega\times(0,h)$ where $\Omega\in\mathbb{R}^{2}$ is a bounded open set with $C^{2}$ boundary. After rescaling space we can assume that $|\Omega|=1$.

Let
$$\textbf{u}^{h}:~\Omega(h)\times[0,T]\rightarrow\mathbb{R}^{3},$$
and
$$\textbf{u}^{h}\in L^{\infty}((0,T);H^{1}(\Omega(h)))\cap L^{2}((0,T);H^{2}(\Omega(h)))\cap \dot{H}^{1}((0,T);L^{\frac{3}{2}}(\Omega(h)))$$
be a family of weak solution of the LLB equation with the initial data $\textbf{u}_{0}$ and  Neumann boundary condition $\frac{\partial\textbf{u}^{h}}{\partial\textbf{n}}=0$:
\begin{equation}\label{4.1}
\frac{\partial\textbf{u}^{h}}{\partial t}=\gamma\textbf{u}^{h}\times\textbf{H}+L\textbf{H},
\end{equation}
where $\textbf{H}=A\Delta\textbf{u}^{h}-\frac{1}{\chi_{11}}\left(1+\mu|\textbf{u}^{h}|^{2}\right)\textbf{u}^{h}-\nabla U^{h}.$

Suppose that the spin magnetic ratio $\nu=\nu(h)$, damping parameter $L=L(h)$ and exchange constant $A=A(h)$ are all functions of $h$.  Longitudinal susceptibility $\chi_{11}$ and $\mu=\frac{3T}{5(T-T_{c})}$ are regarded as constants. The magnetostatic potential $U^{h}$ satisfies the Maxwell equation:
\begin{equation}\label{4.2}
\int_{\mathbb{R}^{3}}\nabla U^{h}\cdot\nabla\phi dx=\int_{\Omega(h)}\textbf{u}^{h}\cdot\nabla\phi dx,\quad \forall\phi\in W^{1,\frac{6}{5}}(\mathbb{R}^{3}).
\end{equation}

\begin{definition}\label{definition4.1}
If for any given finite $T>0$, $\forall\Phi\in L^{2}((0,T);H^{1}(\Omega(h)))$,
\begin{align}\label{4.3}
\int_{0}^{T}\int_{\Omega(h)}\textbf{u}_{t}^{h}\cdot\Phi dxdt=&\int_{0}^{T}\int_{\Omega(h)}\Big\{-\gamma A\left(\textbf{u}^{h}\times\nabla\textbf{u}^{h}\right)\cdot\nabla\Phi-\gamma\left(\textbf{u}^{h}\times\nabla U^{h}\right)\cdot\Phi-LA\nabla\textbf{u}^{h}\cdot\nabla\Phi\notag\\
&-\frac{L}{\chi_{11}}\textbf{u}^{h}\cdot\Phi-\frac{L\mu}{\chi_{11}}\left|\textbf{u}^{h}\right|^{2}\textbf{u}^{h}\cdot\Phi-L\nabla U^{h}\cdot\Phi\Big\}dxdt,
\end{align}
then $\textbf{u}^{h}$ is a weak solution of \eqref{4.1} subject to homogeneous Neumann boundary conditions.
\end{definition}

Assume the initial data satisfies $\sup\limits_{h}\frac{1}{h^{2}}\fint_{\Omega(h)}\left|\textbf{u}^{h}(0)\right|^{2}dx<\infty$ and $\sup\limits_{h}\frac{1}{h\sqrt{h}}\fint_{\Omega(h)}\left|\nabla\textbf{u}^{h}(0)\right|^{2}dx<\infty$, from Lemma \ref{6.10}-Lemma
\ref{6.7}, there are some estimates about $\textbf{u}^{h}$ and $v^{h}=\frac{U^{h}}{h}$:
\begin{align}\label{4.4}
&\sup\limits_{h,t}\fint_{\Omega(h)}\left|\textbf{u}^{h}\right|^{2}dx<\infty,\\
\label{4.5}
&\sup\limits_{h,t}\frac{1}{h\sqrt{h}}\fint_{\Omega(h)}\left|\nabla\textbf{u}^h\right|^{2}dx<\infty, \text{ or } \sup\limits_{h,t}\fint_{\Omega(h)}\left|\nabla\textbf{u}^h\right|^{2}dx<\infty,\\
\label{4.6}
&\sup\limits_{h}\int_{0}^{T}\fint_{\Omega(h)}\left|\textbf{u}^{h}\right|^{4}dxdt<\infty,\\
\label{4.7}
&\sup\limits_{h}\int_{0}^{T}\fint_{\Omega(h)}\left|\Delta\textbf{u}^{h}\right|^{2}dxdt<\infty,\\
\label{4.8}
&\sup\limits_{h}\int_{0}^{T}\left(\fint_{\Omega(h)}\left|\textbf{u}_{t}^{h}\right|^{\frac{3}{2}}dx\right)^{\frac{4}{3}}dt<\infty,\\
\label{4.9}
&\sup\limits_{h}\frac{1}{h}\int_{0}^{T}\int_{\mathbb{R}^{3}}\left|\nabla v^{h}\right|^{2}dxdt<\infty, \text{ or }
\sup\limits_{h}\int_{0}^{T}\int_{\mathbb{R}^{3}}\left|\nabla v^{h}\right|^{2}dxdt<\infty,\\
\label{4.14}
&\sup\limits_{h,t}\fint_{\Omega(h)}\left|\textbf{u}^{h}\right|^{6}dx<\infty,\\ \label{4.15}
&\sup\limits_{h}\int_{0}^{T}\left(\fint_{\Omega(h)}\left|\nabla\textbf{u}^{h}\right|^{6}dx\right)^{\frac{1}{3}}dt<\infty,\\
\label{4.16}
&\sup\limits_{h}\int_{0}^{T}\left(\fint_{\Omega(h)}\left|\nabla\textbf{u}^{h}\right|^{3}dx\right)^{\frac{2}{3}}dt<\infty.
\end{align}
Then lemma \ref{lemma3.1}-lemma \ref{lemma3.5} hold.

Now, we state our main results as follow.
\begin{theorem}\label{theorem4.1}
Suppose that there exist $a,~\varepsilon>0$ such that when $h\rightarrow0$,
\begin{equation}\label{4.10}
\frac{L(h)}{\gamma(h)h}\rightarrow a,\quad\gamma(h)\sqrt{h}\rightarrow1,\quad\frac{A(h)}{h}\rightarrow\varepsilon.
\end{equation}
The weak solutions $\textbf{u}^{h}$ of \eqref{4.1} and corresponding renormalized magnetostatic potentials $v^{h}=\frac{U^{h}}{h}$ satisfy \eqref{4.4}-\eqref{4.16} by assuming initial conditions, then there is $ h=h_{k}\rightarrow0$ such that
\begin{equation}\label{4.11}
\fint_{0}^{h}\textbf{u}^{h}dx_{3}\rightharpoonup\textbf{u} \text{ in } L^{\infty}((0,T);H^{1}(\Omega))\cap L^{2}((0,T);H^{2}(\Omega))\cap \dot{H}^{1}((0,T);L^{\frac{3}{2}}(\Omega)),
\end{equation}
and
\begin{equation}\label{4.12}
v^{h}\rightharpoonup v \text{ in } L^{\infty}((0,T);\dot{H}^{1}(\mathbb{R}^{3}))\cap \dot{H}^{1}((0,T);L^{2}(\mathbb{R}^{3})),
\end{equation}
where $\textbf{u}=(u,0):\Omega\times(0,T)\rightarrow\mathbb{R}^{2}$, $(u, v)$ is a weak solution of
\begin{equation}\label{4.13}
u\wedge\left(\partial_{t}^{2}u-\varepsilon|u|^{2}\Delta'u+|u|^{2}\nabla'v\right)=0,
\end{equation}
where $v=v|_{x_{3}=0}$ and
$$\Delta v=div\left(u\chi_{\Omega}\right)\otimes\delta_{{x_{3}=0}} \text{ in } H^{-1}(\mathbb{R}^{3})~ a.e.~ t.$$
\end{theorem}

\section{Proof of Theorem 3.1}\label{section4}
For the convenience of calculation, it is advisable to assume that $|\Omega|=1$. Consider the simplest parameter satisfying \eqref{4.10} and set
$$\gamma=\frac{1}{\sqrt{h}},\quad L=\sqrt{h},\quad A=h,$$
where $a=1,\;\varepsilon=1$. Inserting the above parameters into the equation \eqref{4.3}, we get
\begin{align*}
\int_{0}^{T}\fint_{\Omega(h)}\textbf{u}_{t}^{h}\cdot\Phi dxdt=&\int_{0}^{T}\fint_{\Omega(h)}\Big\{-\sqrt{h}\left(\textbf{u}^{h}\times\nabla\textbf{u}^{h}\right)\cdot\nabla\Phi-\sqrt{h}\left(\textbf{u}^{h}\times\nabla v^{h}\right)\cdot\Phi-h\sqrt{h}\nabla\textbf{u}^{h}\cdot\nabla\Phi\nonumber\\
&-\frac{\sqrt{h}}{\chi_{11}}\textbf{u}^{h}\cdot\Phi-\frac{\mu\sqrt{h}}{\chi_{11}}\left|\textbf{u}^{h}\right|^{2}\textbf{u}^{h}\cdot\Phi-h\sqrt{h}\nabla v^{h}\cdot\Phi\Big\}dxdt,
\end{align*}
\begin{align}\label{5.1}
\Rightarrow\frac{1}{\sqrt{h}}\int_{0}^{T}\fint_{\Omega(h)}\textbf{u}_{t}^{h}\cdot\Phi dxdt=&\int_{0}^{T}\fint_{\Omega(h)}\Big\{-\left(\textbf{u}^{h}\times\nabla\textbf{u}^{h}\right)\cdot\nabla\Phi-\left(\textbf{u}^{h}\times\nabla v^{h}\right)\cdot\Phi-h\nabla\textbf{u}^{h}\cdot\nabla\Phi\nonumber\\
&-\frac{1}{\chi_{11}}\textbf{u}^{h}\cdot\Phi-\frac{\mu}{\chi_{11}}\left|\textbf{u}^{h}\right|^{2}\textbf{u}^{h}\cdot\Phi-h\nabla v^{h}\cdot\Phi\Big\}dxdt,
\end{align}
for any $\Phi\in L^{2}((0,T);H^{1}(\Omega(h)))$, where $v^{h}=\frac{U^{h}}{h}$.

It's easy to know that $\fint_{0}^{h}\textbf{u}^{h}dx_{3}$ and $v^{h}$ are uniformly bounded in space-time by \eqref{4.4}, \eqref{4.5}, \eqref{4.7}-\eqref{4.9} and lemma \ref{lemma3.5}, then \eqref{4.11} and \eqref{4.12} can be obtained from the weak compactness.

According to the compact embedding Theorem, we have $H^{1}(\Omega)\hookrightarrow\hookrightarrow L^{q}(\Omega)\hookrightarrow L^{\frac{3}{2}}(\Omega)$, $\frac{3}{2}\leq q<\infty$.
Let
$$W=\left\{v\in L^{p_{0}}(0,T;H^{1}(\Omega)),\,v_{t}\in L^{2}(0,T;L^{\frac{3}{2}}(\Omega)),\,1<p_{0}<\infty\right\},$$
then $W\hookrightarrow\hookrightarrow L^{p_{0}}\left(0,T;L^{q}(\Omega)\right)$ from Aubin-Lions Lemma,
and
$$\fint_{0}^{h}\textbf{u}^{h}dx_{3}\rightharpoonup \textbf{u} \text{ in } W.$$
Hence
\begin{align}\label{5.2}
\fint_{0}^{h}\textbf{u}^{h}dx_{3}\rightarrow \textbf{u} \text{ in } L^{p_{0}}(0,T;L^{q}(\Omega)),\; 1<p_{0}<\infty,\,\frac{3}{2}\leq q<\infty.
\end{align}

To prove the limit equation, we consider
$$w_{1}^{h}=\frac{1}{\sqrt{h}}\fint_{0}^{h}\left|\textbf{u}^{h}\right|^{2}u_{3}^{h}dx_{3},\quad w_{2}^{h}=\frac{1}{\sqrt{h}}\fint_{0}^{h}\left|\textbf{u}^{h}\right|^{2}\frac{\partial U^{h}}{\partial x_{3}}dx_{3}.$$
It is known from Section 2 that they are uniformly bounded in $L^{\infty}(0,T;L^{\frac{6}{5}})$, and have the same weak limit. Next, we will obtain their weak limits and the weak limits of their time derivatives in Proposition \ref {5.1} and Proposition \ref {5.2}, respectively. Finally, we derive the limit equation \eqref {4.13}.

\begin{proposition}\label{proposition5.1}
Under the assumptions of Theorem \ref{theorem4.1}, there exists $h=h_{k}\rightarrow0$ such that
$$w_{2}^{h}\rightharpoonup u\wedge u_{t} \text{ in } L^{2}(0,T;L^{\frac{6}{5}}(\Omega)).$$
\end{proposition}

\begin{proof}
The proof of Proposition \ref {5.1} is divided into two steps.

\underline{Step 1}: $w_{2}^{h}\rightharpoonup\fint_{0}^{h}u^{h}\wedge u_{t}^{h}dx_{3} \text{ in } L^{2}(0,T;L^{\frac{6}{5}}(\Omega)).$

Substituting $\Phi=\textbf{u}^{h}_{\perp}\phi=\left(-u_{2}^{h},u^{h}_{1},0\right)\phi$ into the equation \eqref{5.1}, where $\phi\in C_{0}^{\infty}(\Omega\times(0,T))$, we obtain
\begin{align}\label{5.3}
\frac{1}{\sqrt{h}}\int_{0}^{T}\!\!\!\fint_{\Omega(h)}\textbf{u}_{t}^{h}\cdot\textbf{u}^{h}_{\perp}\phi dxdt&=\int_{0}^{T}\!\!\!\fint_{\Omega(h)}\Big\{-\left(\textbf{u}^{h}\times\nabla\textbf{u}^{h}\right)\cdot\nabla\left(\textbf{u}^{h}_{\perp}\phi\right)-\left(\textbf{u}^{h}\times\nabla v^{h}\right)\cdot\textbf{u}^{h}_{\perp}\phi\notag\\
&\quad-h\nabla\textbf{u}^{h}\cdot\textbf{u}^{h}_{\perp}\nabla\phi-h\nabla v^{h}\cdot\textbf{u}^{h}_{\perp}\phi\Big\}dxdt.
\end{align}
If $\textbf{X}\in\mathbb{R}^{3}$,
\begin{align*}
\textbf{u}_{\perp}^{h}\cdot\left(\textbf{u}^{h}\times \textbf{X}\right)&=\left(\textbf{u}^{h}\cdot \textbf{X}\right)u_{3}^{h}-\left|\textbf{u}^{h}\right|^{2}X_{3},
\end{align*}
we have
$$\textbf{u}^{h}_{\perp}\left(\textbf{u}^{h}\times\nabla v^{h}\right)=\left(\textbf{u}^{h}\cdot\nabla v^{h}\right)u_{3}^{h}-\left|\textbf{u}^{h}\right|^{2}\frac{\partial v^{h}}{\partial x_{3}}.$$
Next, we obtain from \eqref{5.3}
\begin{align}\label{5.4}
\int_{0}^{T}\fint_{\Omega(h)}\left(u^{h}\wedge u^{h}_{t}\right)\phi dxdt=&\int^{T}_{0}\fint_{\Omega(h)}\Big\{\sqrt{h}\left(\textbf{u}^{h}\times\Delta\textbf{u}^{h}\right)\cdot\textbf{u}^{h}_{\perp}\phi-\sqrt{h}u^{h}_{3}
\left(\textbf{u}^{h}\cdot\nabla v^{h}\right)\phi\notag\\
&\quad+\frac{1}{\sqrt{h}}\left|\textbf{u}^{h}\right|^{2}\frac{\partial U^{h}}{\partial x_{3}}\phi-h\sqrt{h}\nabla\textbf{u}^{h}\cdot\textbf{u}^{h}_{\perp}\nabla\phi-h\sqrt{h}\nabla v^{h}\cdot\textbf{u}^{h}_{\perp}\phi\Big\}dxdt.
\end{align}
By H\"{o}lder's inequality, the right-hand side of equation \eqref {5.4} can be controlled by
\begin{align}\label{5.5}
&\leq \frac{1}{\sqrt{h}}\int_{0}^{T}\fint_{\Omega(h)}\left|\textbf{u}^{h}\right|^{2}\frac{\partial U^{h}}{\partial x_{3}}\phi dxdt+\left|\int^{T}_{0}\fint_{\Omega(h)}\sqrt{h}\left(\textbf{u}^{h}\times\Delta\textbf{u}^{h}\right)\cdot\textbf{u}^{h}_{\perp}\phi dxdt\right|\notag\\
&\quad+\left|\int^{T}_{0}\fint_{\Omega(h)}\sqrt{h}u^{h}_{3}
\left(\textbf{u}^{h}\cdot\nabla v^{h}\right)\phi dxdt\right|+\left|\int^{T}_{0}\fint_{\Omega(h)}h\sqrt{h}\nabla\textbf{u}^{h}\cdot\textbf{u}^{h}_{\perp}\nabla\phi dxdt\right|+\left|\int^{T}_{0}\fint_{\Omega(h)}h\sqrt{h}\nabla v^{h}\cdot\textbf{u}^{h}_{\perp}\phi dxdt\right|\notag\\
&\lesssim\frac{1}{\sqrt{h}}\int_{0}^{T}\fint_{\Omega(h)}\left|\textbf{u}^{h}\right|^{2}\frac{\partial U^{h}}{\partial x_{3}}\phi dxdt+\sqrt{h}\left(\int_{0}^{T}\fint_{\Omega(h)}\left|\textbf{u}^{h}\right|^{4}dxdt\right)^{\frac{1}{2}}\left(\int_{0}^{T}\fint_{\Omega(h)}\left|\Delta\textbf{u}^{h}\right|^{2}dxdt\right)^
{\frac{1}{2}}\notag\\
&\quad+\sqrt{h}\left|\int_{0}^{T}\fint_{\Omega(h)}\left(\textbf{u}^{h}\cdot\nabla v^{h}\right)u^{h}_{3}dxdt\right|+h\sqrt{h}\left(\int_{0}^{T}\fint_{\Omega(h)}\left|\nabla\textbf{u}^{h}\right|^{2}dxdt\right)^{\frac{1}{2}}\left(\int_{0}^{T}\fint_{\Omega(h)}\left|\textbf{u}^{h}\right|^{2}dxdt\right)
^{\frac{1}{2}}\notag\\
&\quad+h\left(\int_{0}^{T}\int_{\mathbb{R}^{3}}\left|\nabla v^{h}\right|^{2}dxdt\right)^{\frac{1}{2}}\left(\int_{0}^{T}\fint_{\Omega(h)}\left|\textbf{u}^{h}\right|^{2}dxdt\right)^{\frac{1}{2}},
\end{align}
and the third term of \eqref{5.5} is bounded by
\begin{align*}
&\left|\int_{0}^{T}\fint_{\Omega(h)}\left(\textbf{u}^{h}\cdot\nabla v^{h}\right)u^{h}_{3}dxdt\right|\\
&=\frac{1}{h}\left|\int_{0}^{T}\int_{\Omega(h)}\nabla u_{3}^{h}\cdot\textbf{u}^{h}\cdot v^{h}+u_{3}^{h}\cdot div\textbf{u}^{h}\cdot v^{h}dxdt\right|\\
&\lesssim\frac{1}{h^{\frac{1}{6}}}\left(\int_{0}^{T}\fint_{\Omega(h)}\left|\nabla u^{h}_{3}\right|^{2}dxdt\right)^{\frac{1}{2}}\left(\int_{0}^{T}\fint_{\Omega(h)}\left|\textbf{u}^{h}\right|^{3}dxdt\right)^{\frac{1}{3}}\left(\int_{0}^{T}\int_{\mathbb{R}^{3}}\left|v^{h}\right|^{6}dxdt\right)
^{\frac{1}{6}}\\
&\quad+\frac{1}{h^{\frac{1}{6}}}\left(\int_{0}^{T}\fint_{\Omega(h)}\left|u_{3}^{h}\right|^{3}dxdt\right)^{\frac{1}{3}}\left(\int_{0}^{T}\fint_{\Omega(h)}\left|\nabla\textbf{u}^{h}\right|^{2}dxdt\right)
^{\frac{1}{2}}\left(\int_{0}^{T}\int_{\mathbb{R}^{3}}\left|v^{h}\right|^{6}dxdt\right)^{\frac{1}{6}}.
\end{align*}
Using H\"{o}lder's inequality and the G-N inequality, for any fixed $h$, we have
\begin{align*}
\int_{0}^{T}\!\!\!\fint_{\Omega(h)}\left|\textbf{u}^{h}\right|^{3}dxdt\lesssim\left(\int_{0}^{T}\!\!\!\fint_{\Omega(h)}\left|\textbf{u}^{h}\right|^{4}dxdt\right)^{\frac{3}{4}},\;\,
\left(\int_{0}^{T}\!\!\!\int_{\mathbb{R}^{3}}\left|v^{h}\right|^{6}dxdt\right)^{\frac{1}{6}}\lesssim\left(\int_{0}^{T}\!\!\!\int_{\mathbb{R}^{3}}\left|\nabla v^{h}\right|^{2}dxdt\right)^{\frac{1}{2}}.
\end{align*}
By using \eqref{4.5}, \eqref{4.6} and \eqref{4.9}, we can obtain
$$-\sqrt{h}\int_{0}^{T}\fint_{\Omega(h)}\left(\textbf{u}^{h}\cdot\nabla v^{h}\right)u^{h}_{3}dxdt=O(h^{\frac{1}{3}}).$$
Combining \eqref{5.5}, \eqref{4.4}-\eqref{4.7} and \eqref{4.9},  we have
$$\int_{0}^{T}\fint_{\Omega(h)}\left(u^{h}\wedge u^{h}_{t}\right)\phi dxdt=\frac{1}{\sqrt{h}}\int_{0}^{T}\fint_{\Omega(h)}\left|\textbf{u}^{h}\right|^{2}\frac{\partial U^{h}}{\partial x_{3}}\phi dxdt+O(h^{\frac{1}{3}}).$$
Therefore, when $h\rightarrow0$, one has
$$\int_{0}^{T}\fint_{\Omega(h)}\left(u^{h}\wedge u^{h}_{t}\right)\phi dxdt\rightarrow\frac{1}{\sqrt{h}}\int_{0}^{T}\fint_{\Omega(h)}\left|\textbf{u}^{h}\right|^{2}\frac{\partial U^{h}}{\partial x_{3}}\phi dxdt.$$
The conclusion of Step 1 is proved.

\underline{Step 2}: $\fint_{0}^{h}u^{h}\wedge u_{t}^{h}dx_{3}\rightharpoonup u\wedge u_{t} \text{ in } L^{2}(0,T;L^{\frac{6}{5}}(\Omega))$, i.e.,
$\int_{0}^{T}\fint_{\Omega(h)}u^{h}\wedge u^{h}_{t}\cdot\phi dxdt\rightarrow\int_{0}^{T}\int_{\Omega}u\wedge u_{t}\cdot\phi dxdt.$

From Lemma \ref{lemma3.3}, H\"{o}lder's inequality, \eqref{4.5}, and \eqref{4.8}, we obtain as $h\rightarrow0$
\begin{align*}
&\quad\left|\int_{0}^{T}\left[\fint_{\Omega(h)}u^{h}_{1}\partial_{t}u^{h}_{2}\phi dx-\int_{\Omega}\left(\fint_{0}^{h}u_{1}^{h}dx_{3}\right)\left(\fint_{0}^{h}\partial_{t}u_{2}^{h}dx_{3}\right)\phi dx\right]dt\right|\\
&\leq h\left[\int_{0}^{T}\left(\fint_{\Omega(h)}\left|\frac{\partial u_{1}^{h}}{\partial x_{3}}\right|^{3}dx\right)^{\frac{2}{3}}dt\right]^{\frac{1}{2}}\left[\int_{0}^{T}\left(\fint_{\Omega(h)}\left|\partial_{t}u
^{h}_{2}\right|^{\frac{3}{2}}\left|\phi\right|^{\frac{3}{2}}dx\right)^{\frac{4}{3}}dt\right]^{\frac{1}{2}}\rightarrow0,
\end{align*}
that is,
$$\int_{0}^{T}\fint_{\Omega(h)}u^{h}_{1}\partial_{t}u^{h}_{2}\phi dxdt\rightarrow\int_{0}^{T}\int_{\Omega}\left(\fint_{0}^{h}u_{1}^{h}dx_{3}\right)\left(\fint_{0}^{h}\partial_{t}u_{2}^{h}dx_{3}\right)\phi dxdt.$$
Similarly,
$$\int_{0}^{T}\fint_{\Omega(h)}u^{h}_{2}\partial_{t}u^{h}_{1}\phi dxdt\rightarrow\int_{0}^{T}\int_{\Omega}\left(\fint_{0}^{h}u_{2}^{h}dx_{3}\right)\left(\fint_{0}^{h}\partial_{t}u_{1}^{h}dx_{3}\right)\phi dxdt.$$
By virtue of \eqref{4.11} and \eqref{5.2}, we can infer that
$\fint_{0}^{h}\partial_{t}\textbf{u}^{h}dx_{3}\rightharpoonup\textbf{u}_{t} \text{ in } L^{2}(0,T;L^{\frac{3}{2}}(\Omega))$,
$\fint_{0}^{h}\textbf{u}^{h}dx_{3}\rightarrow\textbf{u} \text{ in } L^{2}(0,T;L^{3}(\Omega))$, which
combined with \eqref{4.8} yields that
$$\int_{0}^{T}\int_{\Omega}\left(\fint_{0}^{h}u_{1}^{h}dx_{3}\right)\left(\fint_{0}^{h}\partial_{t}u_{2}^{h}dx_{3}\right)\phi dxdt\rightarrow\int_{0}^{T}\int_{\Omega}u_{1}\partial_{t}u_{2}\phi dxdt,$$
$$\int_{0}^{T}\int_{\Omega}\left(\fint_{0}^{h}u_{2}^{h}dx_{3}\right)\left(\fint_{0}^{h}\partial_{t}u_{1}^{h}dx_{3}\right)\phi dxdt\rightarrow\int_{0}^{T}\int_{\Omega}u_{2}\partial_{t}u_{1}\phi dxdt.$$
Thus
$$\int_{0}^{T}\fint_{\Omega(h)}u^{h}\wedge u^{h}_{t}\cdot\phi dxdt\rightarrow\int_{0}^{T}\int_{\Omega}u\wedge u_{t}\cdot\phi dxdt.$$
The conclusion of Step 2 is proved. Combining the conclusions of Step 1 and Step 2, the proof of  Proposition \ref{proposition5.1} is completed.
\end{proof}

\begin{proposition}\label{proposition5.2}
Under the assumptions of Theorem \ref{theorem4.1}, there exists $h=h_{k}\rightarrow0$ such that
$$\partial_{t}w_{1}^{h}\rightharpoonup|u|^{2}\nabla'(u\wedge\nabla'u)-
|u|^{2}u\wedge\nabla'v \text{ in } H^{-1}(0,T;L^{\frac{6}{5}}(\Omega)).$$
\end{proposition}

\begin{proof}
Noticing that
$$\partial_{t}w_{1}^{h}=\frac{1}{\sqrt{h}}\fint_{0}^{h}2\textbf{u}^{h}\cdot\textbf{u}^{h}_{t}\cdot u^{h}_{3}dx_{3}+\frac{1}{\sqrt{h}}\fint_{0}^{h}\left|\textbf{u}^{h}\right|^{2}\partial_{t}u_{3}^{h}dx_{3},$$
we prove the weak limits of the two terms on the right-hand side of the above equality, respectively.

\underline{Step 1}: $\frac{1}{\sqrt{h}}\fint_{0}^{h}\textbf{u}^{h}\cdot\textbf{u}_{t}^{h}\cdot u_{3}^{h}dx_{3}\rightharpoonup0 \text{ in } H^{-1}(0,T;L^{\frac{6}{5}}(\Omega))$, that is,
for any $\phi\in C_{0}^{\infty}(\Omega\times(0,T))$, $\frac{1}{\sqrt{h}}\int_{0}^{T}\fint_{\Omega(h)}\textbf{u}^{h}\cdot\textbf{u}_{t}^{h}\cdot u_{3}^{h}\phi dxdt\rightarrow0.$

According to Lemma \ref{lemma3.3}, H\"{o}lder's inequality, \eqref{4.8}, \eqref{4.14},  and  \eqref{4.15}, we obtain
\begin{align*}
&\frac{1}{\sqrt{h}}\left|\int_{0}^{T}\left[\int_{\Omega}\left(\fint_{0}^{h}\textbf{u}^{h}\cdot\textbf{u}_{t}^{h}\cdot\phi\cdot u_{3}^{h}dx_{3}-\fint_{0}^{h}\textbf{u}^{h}\cdot\textbf{u}^{h}_{t}\cdot\phi dx_{3}\cdot\fint_{0}^{h}u_{3}^{h}dx_{3}\right)dx\right]dt\right|\\
&\leq\sqrt{h}\sup\limits_{t}\left(\fint_{\Omega(h)}
\left|\textbf{u}^{h}\right|^{6}\left|\phi\right|^{6}dx\right)^{\frac{1}{6}}\left[\int_{0}^{T}\!\!\!\left(\fint_{\Omega(h)}\left|\textbf{u}_{t}^{h}\right|^{\frac{3}{2}}dx\right)^{\frac{4}{3}}dt\right]^{\frac{1}{2}}
\left[\int_{0}^{T}\!\!\!\left(\fint_{\Omega(h)}\left|\frac{\partial u_{3}^{h}}{\partial x_{3}}\right|^{6}dx\right)^{\frac{1}{3}}dt\right]^{\frac{1}{2}}\rightarrow0,
\end{align*}
that is,
$$\frac{1}{\sqrt{h}}\int_{0}^{T}\fint_{\Omega(h)}\textbf{u}^{h}\cdot\textbf{u}_{t}^{h}\cdot u_{3}^{h}dx_{3}\cdot\phi dxdt\rightarrow\frac{1}{\sqrt{h}}\int_{0}^{T}\int_{\Omega}\fint_{0}^{h}\textbf{u}^{h}\cdot\textbf{u}_{t}^{h}dx_{3}\cdot\fint_{0}^{h}u_{3}^{h}dx_{3}
\cdot\phi dxdt.$$
By H\"{o}lder's inequality, we have
\begin{align*}
\quad\int_{0}^{T}\left(\int_{\Omega}\left|\frac{1}{\sqrt{h}}\fint_{0}^{h}\textbf{u}^{h}\cdot\textbf{u}_{t}^{h}dx_{3}\right|^{\frac{6}{5}}dx\right)^{\frac{5}{3}}dt
\leq\left(\sup\limits_{t}\frac{1}{h^{3}}\fint_{\Omega(h)}\left|\textbf{u}^{h}\right|^{6}dx\right)^{\frac{1}{3}}\int_{0}^{T}\left(\fint_{\Omega(h)}\left|\textbf{u}_{t}^{h}\right|
^{\frac{3}{2}}dx\right)^{\frac{4}{3}}dt,
\end{align*}
and using G-N inequality and \eqref{4.5}, we can further get
\begin{align*}
\left(\sup\limits_{h,t}\frac{1}{h^{3}}\fint_{\Omega(h)}\left|\textbf{u}^{h}\right|^{6}dx\right)^{\frac{1}{3}}&\lesssim\sup\limits_{h,t}\frac{1}
{h^{\frac{1}{3}}}\fint_{\Omega(h)}\left|\nabla\textbf{u}^{h}\right|^{2}dx=\sup\limits_{h,t}\left(h^{\frac{13}{6}}\cdot\frac{1}
{h\sqrt{h}}\fint_{\Omega(h)}\left|\nabla\textbf{u}^{h}\right|^{2}dx\right)<\infty.
\end{align*}
We know $\frac{1}{\sqrt{h}}\fint_{0}^{h}\textbf{u}^{h}\cdot\textbf{u}_{t}^{h}dx_{3}$ is uniformly bounded in $L^{2}(0,T;L^{\frac{6}{5}})$ in combination with \eqref{4.8}, and $\fint_{0}^{h}u_{3}^{h}dx_{3}\rightarrow 0$ in $L^{2}(0,T;L^{6}(\Omega))$ by \eqref{5.2}. Then we find that
$$\frac{1}{\sqrt{h}}\int_{0}^{T}\int_{\Omega}\fint_{0}^{h}\textbf{u}^{h}\cdot\textbf{u}_{t}^{h}dx_{3}\cdot\fint_{0}^{h}u_{3}^{h}dx_{3}
\cdot\phi dxdt\rightarrow 0.$$
The conclusion of Step 1 is proved.

\underline{Step 2}: $\frac{1}{\sqrt{h}}\fint_{0}^{h}\left|\textbf{u}^{h}\right|^{2}\partial_{t}u_{3}^{h}dx_{3}\rightharpoonup|u|^{2}\nabla'\left(u\wedge\nabla'u\right)-
|u|^{2}u\wedge\nabla'v \text{ in } H^{-1}(0,T;L^{\frac{6}{5}}(\Omega))$, i.e.,
for any $\phi\in C_{0}^{\infty}(\Omega\times(0,T))$,
$$\frac{1}{\sqrt{h}}\int_{0}^{T}\fint_{\Omega(h)}\left|\textbf{u}^{h}\right|^{2}\partial_{t}u_{3}^{h}\phi dxdt\rightarrow\int_{0}^{T}\int_{\Omega}\left[|u|^{2}\nabla'(u\wedge\nabla'u)-
|u|^{2}u\wedge\nabla'v\right]\phi dxdt.$$
Substitute $\Phi=\left|\textbf{u}^{h}\right|^{2}\hat{e}_{3}\phi$, $\phi\in C_{0}^{\infty}(\Omega\times(0,T))$ into the equation \eqref{5.1}
\begin{align}\label{5.6}
\frac{1}{\sqrt{h}}\int_{0}^{T}\!\!\!\fint_{\Omega\left(h\right)}\left|\textbf{u}^{h}\right|^{2}\partial_{t}u^{h}_{3}\phi dxdt&=\int_{0}^{T}\!\!\!\fint_{\Omega\left(h\right)}\Big\{-\left(u^{h}\wedge\nabla u^{h}\right)\cdot\nabla\left(\left|\textbf{u}^{h}\right|^{2}\phi\right)-\left(u^{h}\wedge\nabla v^{h}\right)\cdot\left|\textbf{u}^{h}\right|^{2}\phi\notag\\
&\quad-h\nabla u_{3}^{h}\cdot\nabla\left(\left|\textbf{u}^{h}\right|^{2}\phi\right)-\frac{1}{\chi_{11}}u_{3}^{h}\cdot\left|\textbf{u}^{h}\right|^{2}\phi-\frac{\mu}{\chi_{11}}\left|\textbf{u}^{h}\right|^{2}
u^{h}_{3}\cdot\left|\textbf{u}^{h}\right|^{2}\phi\notag\\
&\quad-h\frac{\partial v^{h}}{\partial x_{3}}\cdot\left|\textbf{u}^{h}\right|^{2}\phi\Big\}dxdt\notag\\
&=:\int_{0}^{T}\!\!\!\fint_{\Omega\left(h\right)}\left[-\left(u^{h}\wedge\nabla u^{h}\right)\cdot\nabla\left(\left|\textbf{u}^{h}\right|^{2}\phi\right)-\left(u^{h}\wedge\nabla v^{h}\right)\cdot\left|\textbf{u}^{h}\right|^{2}\phi\right]dxdt\notag\\
&\quad+A+B+C+D,
\end{align}
where
$$A=-h\int_{0}^{T}\fint_{\Omega(h)}\nabla u_{3}^{h}\cdot\nabla\left(\left|\textbf{u}^{h}\right|^{2}\phi\right)dxdt,\quad B=-\frac{1}{\chi_{11}}\int_{0}^{T}\int_{\Omega(h)}u_{3}^{h}\cdot\left|\textbf{u}^{h}\right|^{2}\phi dxdt,$$
$$C=-\frac{\mu}{\chi_{11}}\int_{0}^{T}\fint_{\Omega(h)}\left|\textbf{u}^{h}\right|^{4}u_{3}^{h}\phi dxdt,\quad D=-h\int_{0}^{T}\fint_{\Omega(h)}\frac{\partial v^{h}}{\partial x_{3}}\cdot\left|\textbf{u}^{h}\right|^{2}\phi dxdt.$$

Next, we consider the limiting problem of the right-hand side of \eqref{5.6} when $h\rightarrow 0$. For the term $A$, applying H\"{o}lder's inequality, \eqref{4.6},  and  \eqref{4.7}, we obtain
\begin{align*}
\left|\int_{0}^{T}\!\!\!\fint_{\Omega(h)}\nabla u_{3}^{h}\cdot\nabla\left(\left|\textbf{u}^{h}\right|^{2}\phi\right)dxdt\right|&=\left|\int_{0}^{T}\!\!\!\fint_{\Omega(h)}\Delta u^{h}_{3}\cdot\left|\textbf{u}^{h}\right|^{2}\phi dxdt\right|\\
&\leq\left(\int_{0}^{T}\!\!\!\fint_{\Omega(h)}\left|\Delta u_{3}^{h}\right|^{2}dxdt\right)^{\frac{1}{2}}\left(\int_{0}^{T}\!\!\!\fint_{\Omega(h)}\left|\textbf{u}^{h}\right|^{4}\left|\phi\right|^{2}dxdt\right)
^{\frac{1}{2}}<\infty,\quad \forall h,
\end{align*}
then $A\rightarrow 0.$

For the term $B$,  employing  Lemma \ref{lemma3.3}, H\"{o}lder's inequality, \eqref{4.5},  and  \eqref{4.6}, we have
\begin{align*}
&\quad\left|\int_{0}^{T}\left[\int_{\Omega}\left(\fint_{0}^{h}u_{3}^{h}\cdot\left|\textbf{u}^{h}\right|^{2}\phi
dx_{3}-\fint_{0}^{h}u_{3}^{h}dx_{3}\cdot\fint_{0}^{h}\left|\textbf{u}^{h}\right|^{2}\phi dx_{3}\right)dx\right]dt\right|\\
&\leq h\left(\int_{0}^{T}\fint_{\Omega(h)}\left|\frac{\partial u_{3}^{h}}{\partial x_{3}}\right|^{2}dxdt\right)^{\frac{1}{2}}\left(\int_{0}^{T}\fint_{\Omega(h)}\left|\textbf{u}^{h}\right|^{4}|\phi|^{2}dxdt\right)^{\frac{1}{2}}\rightarrow 0,
\end{align*}
that is,
$$\int_{0}^{T}\int_{\Omega(h)}u_{3}^{h}\cdot\left|\textbf{u}^{h}\right|^{2}\phi dxdt\rightarrow\int_{0}^{T}\int_{\Omega}\fint_{0}^{h}u_{3}^{h}dx_{3}\cdot\fint_{0}^{h}\left|\textbf{u}^{h}\right|^{2}\phi dx_{3}dxdt.$$
We utilize H\"{o}lder's inequality, \eqref{4.6} and \eqref{5.2} to get
$$\int_{0}^{T}\int_{\Omega}\fint_{0}^{h}u_{3}^{h}dx_{3}\cdot\fint_{0}^{h}\left|\textbf{u}^{h}\right|^{2}\phi dx_{3}dxdt\rightarrow 0,$$
which yields $B\rightarrow 0$.

For the term $C$, using Lemma \ref{lemma3.3}, H\"{o}lder's inequality, \eqref{4.14}, and \eqref{4.16}, we deduce
\begin{align*}
&\quad\left|\int_{0}^{T}\left[\int_{\Omega}\left(\fint_{0}^{h}\left|\textbf{u}^{h}\right|^{4}u_{3}^{h}\phi dx_{3}-\fint_{0}^{h}\left|\textbf{u}^{h}\right|^{4}\phi dx_{3}\cdot\fint_{0}^{h}u_{3}^{h}dx_{3}\right)dx\right]dt\right|\\
&\leq h\left[\int_{0}^{T}\left(\fint_{\Omega(h)}\left|\textbf{u}^{h}\right|^{6}|\phi|^{\frac{3}{2}}dx\right)^{\frac{4}{3}}dt\right]^{\frac{1}{2}}\left[\int_{0}^{T}\left(\fint_{\Omega(h)}
\left|\frac{\partial u_{3}^{h}}{\partial x_{3}}\right|^{3}dx\right)^{\frac{2}{3}}dt\right]^{\frac{1}{2}}\rightarrow 0,
\end{align*}
that is,
$$\int_{0}^{T}\fint_{\Omega(h)}\left|\textbf{u}^{h}\right|^{4}u_{3}^{h}\phi dxdt\rightarrow\int_{0}^{T}\int_{\Omega}\fint_{0}^{h}\left|\textbf{u}^{h}\right|^{4}\phi dx_{3}\cdot\fint_{0}^{h}u_{3}^{h}dx_{3}dxdt.$$
Together with \eqref{5.2} and \eqref{4.14}, we can easily get that
$$\int_{0}^{T}\int_{\Omega}\fint_{0}^{h}\left|\textbf{u}^{h}\right|^{4}\phi dx_{3}\cdot\fint_{0}^{h}u_{3}^{h}dx_{3}dxdt\rightarrow 0,$$
that is, $C\rightarrow 0$.

For the term $D$, we have
$$|D|=\sqrt{h}\cdot\sqrt{h}\left|\int_{0}^{T}\fint_{\Omega(h)}\frac{\partial v^{h}}{\partial x_{3}}\cdot\left|\textbf{u}^{h}\right|^{2}\phi dxdt\right|.$$
By H\"{o}lder's  inequality, \eqref{4.6} and \eqref{4.9}, one has
\begin{align*}
\sqrt{h}\left|\int_{0}^{T}\fint_{\Omega(h)}\frac{\partial v^{h}}{\partial x_{3}}\cdot\left|\textbf{u}^{h}\right|^{2}\phi dxdt\right|\leq\left(\int_{0}^{T}\int_{\mathbb{R}^{3}}\left|\frac{\partial v^{h}}{\partial x_{3}}\right|^{2}dxdt\right)^{\frac{1}{2}}\left(\int_{0}^{T}\fint_{\Omega(h)}\left|\textbf{u}^{h}\right|^{4}|\phi|^{2}dxdt\right)^{\frac{1}{2}}
<\infty,\quad \forall h,
\end{align*}
it follows that $D\rightarrow 0.$

Inserting the above estimates into \eqref{5.6}, we get
$$\frac{1}{\sqrt{h}}\int_{0}^{T}\!\!\!\fint_{\Omega(h)}\left|\textbf{u}^{h}\right|^{2}\partial_{t}u^{h}_{3}\phi dxdt\rightarrow\int_{0}^{T}\!\!\!\fint_{\Omega(h)}\left[-\left(u^{h}\wedge\nabla u^{h}\right)\cdot\nabla\left(\left|\textbf{u}^{h}\right|^{2}\phi\right)-\left(u^{h}\wedge\nabla v^{h}\right)\cdot\left|\textbf{u}^{h}\right|^{2}\phi\right]dxdt.$$

Now, we  prove
\begin{align}\label{5.12}
&\int_{0}^{T}\fint_{\Omega(h)}\left[-\left(u^{h}\wedge\nabla u^{h}\right)\cdot\nabla\left(\left|\textbf{u}^{h}\right|^{2}\phi\right)-\left(u^{h}\wedge\nabla v^{h}\right)\cdot\left|\textbf{u}^{h}\right|^{2}\phi\right]dxdt\notag\\
&\quad\rightarrow\int_{0}^{T}\int_{\Omega}\left[\left(u\wedge\Delta'u\right)\cdot\left|u\right|^{2}\phi-
\left(u\wedge\nabla'v\right)\cdot\left|u\right|^{2}\phi\right]dxdt.
\end{align}
First, we prove that
$$-\int_{0}^{T}\fint_{\Omega(h)}\left(u^{h}\wedge\nabla u^{h}\right)\cdot\nabla\left(\left|\textbf{u}^{h}\right|^{2}\phi\right)dxdt\rightarrow
\int_{0}^{T}\int_{\Omega}\left(u\wedge\Delta'u\right)\cdot\left|u\right|^{2}\phi dxdt.$$
By virtue of Lemma \ref{lemma3.3}, H\"{o}lder's inequality, \eqref{4.7}, \eqref{4.14}, and \eqref{4.15}, we obtain
\begin{align}\label{5.14}
&\quad\left|\int_{0}^{T}\int_{\Omega}\left(\fint_{0}^{h}u_{1}^{h}\Delta u_{2}^{h}\left|\textbf{u}^{h}\right|^{2}\phi dx_{3}-\fint_{0}^{h}u_{1}^{h}dx_{3}\cdot\fint_{0}^{h}\Delta u_{2}^{h}\left|\textbf{u}^{h}\right|^{2}\phi dx_{3}\right)dxdt\right|\notag\\
&\lesssim h\left(\sup\limits_{t}\fint_{\Omega(h)}\left|\textbf{u}^{h}\right|^{6}dx\right)^{\frac{1}{3}}\left[\int_{0}^{T}\left(\fint_{\Omega(h)}\left|\frac{\partial u_{1}^{h}}{\partial x_{3}}\right|^{6}dx\right)^{\frac{1}{3}}dt\right]^{\frac{1}{2}}\left(\int_{0}^{T}\fint_{\Omega(h)}\left|\Delta u_{2}^{h}\right|^{2}dxdt\right)^{\frac{1}{2}}\rightarrow 0,
\end{align}
that is,
$$\int_{0}^{T}\fint_{\Omega(h)}u_{1}^{h}\Delta u_{2}^{h}\left|\textbf{u}^{h}\right|^{2}\phi dxdt\rightarrow\int_{0}^{T}\int_{\Omega}\left(\fint_{0}^{h}u_{1}^{h}dx_{3}\right)\cdot\left(\fint_{0}^{h}\Delta u_{2}^{h}\left|\textbf{u}^{h}\right|^{2}\phi dx_{3}\right)dxdt.$$
By \eqref{5.2}, we have
\begin{align}\label{5.7}
\fint_{0}^{h}u_{1}^{h}dx_{3}\rightarrow u_{1} \text{ in } L^{6}(0,T;L^{6}(\Omega)).
\end{align}
Using Lemma \ref{lemma3.3}, H\"{o}lder's inequality, \eqref{4.7}, \eqref{4.14}, and \eqref{4.16}, we can also get
\begin{align}\label{5.15}
&\quad\left|\int_{0}^{T}\int_{\Omega}\left(\fint_{0}^{h}\Delta u_{2}^{h}\left|\textbf{u}^{h}\right|^{2}dx_{3}-\fint_{0}^{h}\Delta u_{2}^{h}dx_{3}\cdot\fint_{0}^{h}\left|\textbf{u}^{h}\right|^{2}dx_{3}\right)dxdt\right|\notag\\
&\lesssim h\left(\int_{0}^{T}\fint_{\Omega(h)}\left|\Delta u_{2}^{h}\right|^{2}dxdt\right)^{\frac{1}{2}}\sup\limits_{t}\left(\fint_{\Omega(h)}\left|\textbf{u}^{h}\right|^{6}dx\right)^{\frac{1}{6}}\left
[\int_{0}^{T}\left(\fint_{\Omega(h)}\left|\frac{\partial\textbf{u}^{h}}{\partial x_{3}}\right|^{3}dx\right)^{\frac{2}{3}}dt\right]^{\frac{1}{2}}\rightarrow 0,
\end{align}
that is,
$$\int_{0}^{T}\fint_{\Omega(h)}\Delta u_{2}^{h}\left|\textbf{u}^{h}\right|^{2}dx_{3}dxdt\rightarrow\int_{0}^{T}\int_{\Omega}\left(\fint_{0}^{h}\Delta u_{2}^{h}dx_{3}\right)\cdot\left(\fint_{0}^{h}\left|\textbf{u}^{h}\right|^{2}dx_{3}\right)dxdt.$$
By \eqref{4.11}, we have
\begin{align}\label{5.8}
\fint_{0}^{h}\Delta u_{2}^{h}dx_{3}\rightharpoonup\Delta u_{2} \text{ in }  L^{2}(0,T;L^{2}(\Omega)).
\end{align}
Let $W=\left\{v\in L^{3}(0,T;W^{1,\frac{3}{2}}(\Omega)),v_{t}\in L^{2}(0,T;L^{\frac{6}{5}}(\Omega))\right\}$, by Aubin-Lions Lemma and the compact embedding theorem, we conclude that
$$W\hookrightarrow\hookrightarrow L^{3}(0,T;L^{3}(\Omega)).$$
Through simple calculation and using H\"{o}lder's inequality, \eqref{4.5}, \eqref{4.8}, \eqref{4.14}, we can know that $\fint_{0}^{h}\left|\textbf{u}^{h}\right|^{2}dx_{3}$ is uniformly bounded about $h$ in $W$. Then $\fint_{0}^{h}\left|\textbf{u}^{h}\right|^{2}dx_{3}$ converges weakly in $W$ by the weak compactness argument.

Utilizing Lemma \ref{lemma3.3}, \eqref{4.4}, and \eqref{4.5}, we obtain
\begin{align*}
&\left|\int_{0}^{T}\int_{\Omega}\left[\fint_{0}^{h}\left|\textbf{u}^{h}\right|^{2}dx_{3}-\left(\fint_{0}^{h}\textbf{u}^{h}dx_{3}\right)^{2}\right]\varphi dxdt\right|\\
&\leq h\left(\int_{0}^{T}\fint_{\Omega(h)}\left|\textbf{u}^{h}\right|^{2}|\varphi|^{2}dxdt\right)^{\frac{1}{2}}\left(\int_{0}^{T}\fint_{\Omega(h)}
\left|\frac{\partial\textbf{u}^{h}}
{\partial x_{3}}\right|^{2}dxdt\right)^{\frac{1}{2}}\\
&\rightarrow 0.
\end{align*}
Hence, $\fint_{0}^{h}\left|\textbf{u}^{h}\right|^{2}dx_{3}\rightharpoonup\left(\fint_{0}^{h}\textbf{u}^{h}dx_{3}\right)^{2} \text{ in } W$.  By H\"{o}lder's inequality, we have
\begin{align*}
&\quad\left|\int_{0}^{T}\!\!\int_{\Omega}\left[\left(\fint_{0}^{h}\textbf{u}^{h}dx_{3}\right)^{2}-|\textbf{u}|^{2}\right]\phi dxdt\right|=\left|\int_{0}^{T}\!\!\int_{\Omega}\left(\fint_{0}^{h}\textbf{u}^{h}dx_{3}-\textbf{u}\right)\left(\fint_{0}^{h}\textbf{u}^{h}dx_{3}+\textbf{u}\right)\phi dxdt\right|\\
&\leq\left(\int_{0}^{T}\!\!\int_{\Omega}\left|\fint_{0}^{h}\textbf{u}^{h}dx_{3}-\textbf{u}\right|^{2}dxdt\right)^{\frac{1}{2}}\left(\int_{0}^{T}\!\!\int_{\Omega}\left|\fint_{0}^{h}\textbf{u}^{h}dx_{3}
+\textbf{u}\right|^{2}|\phi|^{2}dxdt\right)^{\frac{1}{2}},
\end{align*}
which implies that $\left(\fint_{0}^{h}\textbf{u}^{h}dx_{3}\right)^{2}\rightharpoonup|\textbf{u}|^{2} \text{ in } W$  by \eqref{5.2}.
Therefore, we have
$$\fint_{0}^{h}\left|\textbf{u}^{h}\right|^{2}dx_{3}\rightharpoonup|\textbf{u}|^{2} \text{ in } W,$$
and further get
\begin{align}\label{5.9}
\fint_{0}^{h}\left|\textbf{u}^{h}\right|^{2}dx_{3}\rightarrow|\textbf{u}|^{2} \text{ in } L^{3}(0,T;L^{3}(\Omega)).
\end{align}
It can be obtained by combining \eqref{5.8} and \eqref{5.9} that
\begin{align}\label{5.10}
\fint_{0}^{h}\Delta u_{2}^{h}\left|\textbf{u}^{h}\right|^{2}dx_{3}\rightharpoonup\Delta u_{2}\left|\textbf{u}\right|^{2}\text{ in } L^{\frac{6}{5}}(0,T;L^{\frac{6}{5}}(\Omega)).
\end{align}
Utilizing \eqref{5.7} and \eqref{5.10}, one has
$$\int_{0}^{T}\int_{\Omega}\fint_{0}^{h}u_{1}^{h}dx_{3}\cdot\fint_{0}^{h}\Delta u_{2}^{h}\left|\textbf{u}^{h}\right|^{2}\phi dx_{3}dxdt\rightarrow\int_{0}^{T}\int_{\Omega}u_{1}\Delta u_{2}|\textbf{u}|^{2}\phi dxdt,$$
i.e.,
$$\int_{0}^{T}\fint_{\Omega(h)}u_{1}^{h}\Delta u_{2}^{h}\left|\textbf{u}^{h}\right|^{2}\phi dxdt\rightarrow\int_{0}^{T}\int_{\Omega}u_{1}\Delta u_{2}|\textbf{u}|^{2}\phi dxdt.$$
Similarly, we obtain
$$\int_{0}^{T}\fint_{\Omega(h)}u_{2}^{h}\Delta u_{1}^{h}\left|\textbf{u}^{h}\right|^{2}\phi dxdt\rightarrow\int_{0}^{T}\int_{\Omega}u_{2}\Delta u_{1}|\textbf{u}|^{2}\phi dxdt.$$
Noticing that
\begin{align*}
-\int_{0}^{T}\fint_{\Omega(h)}\left(u^{h}\wedge\nabla u^{h}\right)\cdot\nabla\left(\left|\textbf{u}^{h}\right|^{2}\phi\right)dxdt=
\int_{0}^{T}\fint_{\Omega(h)}\left(u^{h}\wedge\Delta u^{h}\right)\cdot\left|\textbf{u}^{h}\right|^{2}\phi dxdt
\end{align*}
and $|\textbf{u}|^{2}=|u|^{2}$, we get
$$-\int_{0}^{T}\fint_{\Omega(h)}\left(u^{h}\wedge\nabla u^{h}\right)\cdot\nabla\left(\left|\textbf{u}^{h}\right|^{2}\phi\right)
dxdt\rightarrow
\int_{0}^{T}\int_{\Omega}(u\wedge\Delta'u)\cdot|u|^{2}\phi dxdt.$$
Similar to \eqref{5.12}, we will prove
\begin{align}\label{5.13}
\int_{0}^{T}\fint_{\Omega(h)}\left(u^{h}\wedge\nabla v^{h}\right)\cdot\left|\textbf{u}^{h}\right|^{2}\phi dxdt\rightarrow\int_{0}^{T}\int_{\Omega}u\wedge\nabla'v\cdot|u|^{2}\phi dxdt.
\end{align}
According to Lemma \ref{lemma3.3}, H\"{o}lder's inequality, \eqref{4.9}, \eqref{4.14} and \eqref{4.15}, we obtain
\begin{align*}
&\left|\int_{0}^{T}\int_{\Omega}\left(\fint_{0}^{h}u_{1}^{h}\frac{\partial v^{h}}{\partial x_{2}}\left|\textbf{u}^{h}\right|^{2}\phi dx_{3}-\fint_{0}^{h}u_{1}^{h}dx_{3}\cdot\fint_{0}^{h}\frac{\partial v^{h}}{\partial x_{2}}\left|\textbf{u}^{h}\right|^{2}\phi dx_{3}\right)dxdt\right|\notag\\
&\lesssim\sqrt{h}\left(\sup\limits_{t}\fint_{\Omega(h)}\left|\textbf{u}^{h}\right|^{6}dx\right)^{\frac{1}{3}}\left[\int_{0}^{T}
\left(\fint_{\Omega(h)}\left|\frac{\partial u_{1}^{h}}{\partial x_{3}}\right|^{6}dx\right)^{\frac{1}{3}}dt\right]^{\frac{1}{2}}\left(\int_{0}^{T}\int_{\mathbb{R}^{3}}\left|\frac{\partial v^{h}}{\partial x_{2}}\right|^{2}dxdt\right)^{\frac{1}{2}}\rightarrow 0,
\end{align*}
then
$$\int_{0}^{T}\fint_{\Omega(h)}u_{1}^{h}\frac{\partial v^{h}}{\partial x_{2}}\left|\textbf{u}^{h}\right|^{2}\phi dxdt\rightarrow\int_{0}^{T}\int_{\Omega}\fint_{0}^{h}u_{1}^{h}dx_{3}\cdot\fint_{0}^{h}\frac{\partial v^{h}}{\partial x_{2}}\left|\textbf{u}^{h}\right|^{2}\phi dx_{3}dxdt.$$
Similarly, by Lemma \ref{lemma3.3}, H\"{o}lder's inequality, \eqref{4.9}, \eqref{4.14} and \eqref{4.16}, we get
\begin{align*}
&\quad\left|\int_{0}^{T}\int_{\Omega}\left(\fint_{0}^{h}\frac{\partial v^{h}}{\partial x_{2}}\left|\textbf{u}^{h}\right|^{2}dx_{3}-\fint_{0}^{h}\frac{\partial v^{h}}{\partial x_{2}}dx_{3}\cdot\fint_{0}^{h}\left|\textbf{u}^{h}\right|^{2}dx_{3}\right)dxdt\right|\notag\\
&\lesssim\sqrt{h}\left(\int_{0}^{T}\int_{\mathbb{R}^{3}}\left|\frac{\partial v^{h}}{\partial x_{2}}\right|^{2}dxdt\right)^{\frac{1}{2}}\sup\limits_{t}\left(\fint_{\Omega(h)}\left|\textbf{u}^{h}\right|^{6}dx\right)^{\frac{1}{6}}\left
[\int_{0}^{T}\left(\fint_{\Omega(h)}\left|\frac{\partial\textbf{u}^{h}}{\partial x_{3}}\right|^{3}dx\right)^{\frac{2}{3}}dt\right]^{\frac{1}{2}}\rightarrow 0.
\end{align*}
Then
$$\int_{0}^{T}\fint_{\Omega(h)}\frac{\partial v^{h}}{\partial x_{2}}\left|\textbf{u}^{h}\right|^{2}dx_{3}dxdt\rightarrow\int_{0}^{T}\int_{\Omega}\fint_{0}^{h}\frac{\partial v^{h}}{\partial x_{2}}dx_{3}\cdot\fint_{0}^{h}\left|\textbf{u}^{h}\right|^{2}dx_{3}dxdt.$$

Now, we prove
\begin{align}\label{5.16}
\fint_{0}^{h}\nabla v^{h}dx_{3}\rightharpoonup\nabla v(x_{1},x_{2},0,t) \text{ in } L^{2}(0,T;L^{2}(\Omega)).
\end{align}
It follows from H\"{o}lder's inequality and  \eqref{4.9} that $\fint_{0}^{h}\nabla v^{h}dx_{3}$ is uniformly bounded in $L^{2}(0,T;L^{2}(\Omega))$.
Then $\fint_{0}^{h}\nabla v^{h}dx_{3}$ weakly converges in $L^{2}(0,T;L^{2}(\Omega))$ by the weak compactness argument.
From  Proposition 4.2 in \cite{CM}, we deduce $\fint_{0}^{h}v^{h}dx_{3}\rightarrow v(x_{1},x_{2},0,t) \text{ in } L^{2}(0,T;L^{2}(\Omega))$.
By the fact that $L^{2}(\Omega\times(0,T))\subset H^{-1}(\Omega\times(0,T))$, we infer that $\fint_{0}^{h}\nabla v^{h}dx_{3}$ weakly converges in $H^{-1}(\Omega\times(0,T))$.
Then for any $\psi\in H_{0}^{1}(\Omega\times(0,T))$, we have
\begin{align*}
&\int_{0}^{T}\int_{\Omega}\fint_{0}^{h}\nabla v^{h}dx_{3}\cdot\psi dxdt=-\int_{0}^{T}\int_{\Omega}\fint_{0}^{h}v^{h}dx_{3}\cdot\nabla\psi dxdt\\
&\rightarrow -\int_{0}^{T}\int_{\Omega}v(x_{1},x_{2},0,t)\cdot\nabla\psi dxdt=\int_{0}^{T}\int_{\Omega}\nabla v(x_{1},x_{2},0,t)\cdot\psi dxdt,
\end{align*}
thus \eqref{5.16} holds by the uniqueness of the limit. This combined with \eqref{5.9} yields that
$$\fint_{0}^{h}\frac{\partial v^{h}}{\partial x_{2}}\left|\textbf{u}^{h}\right|^{2}dx_{3}\rightharpoonup\frac{\partial v}{\partial x_{2}}|\textbf{u}|^{2} \text{ in } L^{\frac{6}{5}}(0,T;L^{\frac{6}{5}}(\Omega)).$$
It follows from \eqref{5.7} that
$$\int_{0}^{T}\fint_{\Omega(h)}u_{1}^{h}\frac{\partial v^{h}}{\partial x_{2}}\left|\textbf{u}^{h}\right|^{2}\phi dxdt\rightarrow\int_{0}^{T}\int_{\Omega}u_{1}\frac{\partial v}{\partial x_{2}}|\textbf{u}|^{2}\phi dxdt.$$
Similarly,
$$\int_{0}^{T}\fint_{\Omega(h)}u_{2}^{h}\frac{\partial v^{h}}{\partial x_{1}}\left|\textbf{u}^{h}\right|^{2}\phi dxdt\rightarrow\int_{0}^{T}\int_{\Omega}u_{2}\frac{\partial v}{\partial x_{1}}|\textbf{u}|^{2}\phi dxdt.$$
Therefore, \eqref{5.13} is proved by $|\textbf{u}|^{2}=|u|^{2}$, then \eqref{5.12} holds. The conclusion of Step 2 is proved. Combining the conclusions of Step 1 and Step 2, the proof of  Proposition \ref{proposition5.2} is completed.
\end{proof}

Finally, we use Proposition \ref{proposition5.1} and Proposition \ref{proposition5.2} to derive the limit equation \eqref {4.13}.
\begin{proof}
From Lemma \ref{lemma3.2} and Proposition \ref{proposition5.1}, we have
$$w_{1}^{h}\rightharpoonup u\wedge u_{t} \text{ in } L^{2}(0,T;L^{\frac{6}{5}}(\Omega)).$$
This combined with Proposition \ref{proposition5.2} yields that
\begin{align*}
\int_{0}^{T}\int_{\Omega}w_{1}^{h}\cdot\phi_{t}+\partial_{t}w_{1}^{h}\cdot\phi dxdt&\rightarrow\int_{0}^{T}\int_{\Omega}u\wedge u_{t}\cdot\phi_{t}+|u|^{2}\left[\nabla'\left(u
\wedge\nabla'u\right)-u\wedge\nabla'v\right]\cdot\phi dxdt\\
&=\int_{0}^{T}\int_{\Omega}-u\wedge u_{tt}\cdot\phi+|u|^{2}\left[\nabla'\left(u
\wedge\nabla'u\right)-u\wedge\nabla'v\right]\cdot\phi dxdt\\
&=-\int_{0}^{T}\int_{\Omega}u\wedge\left(u_{tt}-|u|^{2}\Delta'u+|u|^{2}\nabla'v\right)\cdot\phi dxdt.
\end{align*}
Since
$$\int_{0}^{T}\int_{\Omega} w_{1}^{h}\cdot\phi_{t}+\partial_{t}w_{1}^{h}\cdot\phi dxdt=\int_{0}^{T}\int_{\Omega} w_{1}^{h}\phi_{t}-w_{1}^{h}\cdot\phi_{t}dxdt=0,$$
then
$$\int_{0}^{T}\int_{\Omega}u\wedge\left(u_{tt}-|u|^{2}\Delta'u+|u|^{2}\nabla'v\right)\cdot\phi dxdt=0.$$
We get the desired results. The proof of Theorem \ref{theorem4.1} is completed.
\end{proof}
\smallskip \noindent {\bf Acknowledgments.}   H. Wang's research was supported by the National Natural Science Foundation of China (No. 11901066), the Natural Science Foundation of Chongqing (No. cstc2019jcyj-msxmX0167), projects No.2019CDXYST0015 and No.2020CDJQY-A040 supported by the Fundamental Research Funds for the Central Universities.

\end{CJK*}

\begin{thebibliography}{aa}
\footnotesize
\bibitem{LDE}
L. Landau, E. Lifshitz. On the theory of the dispersion of magnetic permeability in ferromagnetic bodies, Physikalische Zeitschrift der Sowjetunion , 8(2): 153-169, 1935.

\bibitem{JDJ}
J.D. Jackson. Classical electrodynamics, John Wiley and Sons Inc., New York, second edition, 1975.

\bibitem{CJGC}
C.J. Garc\'{\i}a-Cervera. Numerical micromagnetics: a review, 2007.

\bibitem{FA}
F. Alouges, A. Soyeur. On global weak solutions for Landau-Lifshitz equations:Existence and nouniqueness, Nonlinear Analysis: Theory, Methods$\&$Applications, 18(11): 1071-1084, 1992.

\bibitem{MC}
C. Melcher. Existence of partially regular solutions for Landau-Lifshitz equations in $\mathbb{R}^{3}$, Communications in Partial Differential Equations, 30(4): 567-587, 2005.

\bibitem{GP}
G. Carbou, P. Fabrie. Regular solution for Landau-Lifschitz equations in a bounded domain, Differential and Integral Equations, 14(2): 213-229, 2001.

\bibitem{LXG}
X.G. Liu. Elliptic partial differential equation, Beijing: Higher Education Press, 2015.

\bibitem{GT}
T.L. Gilbert. A phenomenological theory of damping in ferromagnetic materials, Magnetics, IEEE transactions on magnetics, 40(6): 3343-3449, 2004.

\bibitem{CI}
I. Cimr\'{a}k. A survey on the numerics and computations for the Landau-Lifshitz equation of micromagnetism, Archives of Computational Methods in Engineering, 15(3): 1-37, 2007.

\bibitem{KP}
M. Kruz\'{\i}k, A. Prohl. Recent developments in the modeling, analysis, and numerics of ferromagnetism, SIAM review, 48(3): 439-483, 2006.

\bibitem{MT}
M. Feischl, T. Tran. Existence of regular solutions of the Landau-Lifshitz-Glibert equation in 3D with natural boundary conditions, SIAM Journal on Mathematical Analysis, 49(6): 4470-4490, 2017.

\bibitem{GJ}
G. Gioia, R.D. James. Micromagnetics of very thin films, Proceedings of the Royal Society of London. Series A: Mathematical, Physical and Engineering Sciences, 453(1956): 213-223, 1997.

\bibitem{KS}
R.V. Kohn, V.V. Slastikov. Another thin film limit in micromagnetics. Archive for rational mechanics and analysis, 178(2): 227-245, 2005.

\bibitem{DKM}
A. DeSimone, R.V. Kohn, S. M\"{u}ller, F. Otto. A reduced theory for thin-film micromagnetics, Communications on Pure and Applied Mathematics: A Journal Issued by the Courant Institute of Mathematical Sciences, 55(11): 1408-1460, 2002.

\bibitem{HT}
M. Hadda, M. Tilioua. Thin film limits in magnetoalastic interactions, Mathematical Problems in Engineering, 2012, 2012.

\bibitem{WGC}
C.J. Garcia-Cervera, J. Carlos, E. Weinan. Effective dynamics for ferromagnetic thin films, Journal of Applied Physics, 90(1): 370-374, 2001.

\bibitem{MR}
R. Moser. Boundary vortices for thin ferromagnetic films, Archive for rational mechanics and analysis, 174(2): 267-300, 2004.

\bibitem{CMO}
A. Capella, C. Melcher, F. Otto. Wave-type dynamics in ferromagnetic thin films and the motion of N\'{e}el walls, Nonlinearity, 20(11): 2519-2537, 2007.

\bibitem{CM}
C. Melcher. Thin-film limits for Landau-Lifshitz-Gilbert equation, SIAM Journal on Mathematical Analysis, 42(1): 519-537, 2010.

\bibitem{GDA1}
D.A. Garanin. Generalized equation of motion for a ferromagnet, Physica A: Statistical Mechanics and its Applications, 172(3): 470-491, 1991.

\bibitem{GDA2}
D.A. Garanin. Fokker-Planck and Landau-Lifshitz-Bloch equations for classical ferromagnets, Physical Review B, 55(5): 3050-3057, 1997.

\bibitem{KNL}
K.N. Le. Weak solution of the Landau-Lifshitz-Bloch equation, Journal of Differential Equations, 261(12): 6699-6717, 2016.

\bibitem{ZJ}
Z. Jia. Local strong solution to general Landau-Lifshitz-Bloch equation, arXiv preprint arXiv: 1802.00144, 2018.

\bibitem{GL}
B.L. Guo, Q. Li, M. Zeng. Global smooth solutions of the Landau-Lifshitz-Bloch equation, Preprint.

\bibitem{GBL}
B.L. Guo, Y.Q. Han, D.W. Huang. Weak and smooth global solution for Landau-Lifshitz-Bloch-Maxwell equation, Ann. of Appl. Math. 36(1): 1-30, 2020.

\bibitem{LARE}
L.J. Atkinson, R.F.L. Evans, R.W. Chantrell. Micromagnetic modeling of the heat-assisted switching process in high anisotropy FePt granular thin films, Journal of Applied Physics, 128(7): 073907, 2020.

\bibitem{SA}
M. Sultan, U. Atxitia, A. Melnikov. Electron- and phonon-mediated ultrafast magnetization dynamics of Gd(0001), Physical Review B, 85(18): 184407, 2012.

\bibitem{HN}
D. Hinzke, U. Nowak. Domain wall motion by the magnonic spin Seebeck effect, Physical Review Letters, 107(2): 027205, 2011.

\bibitem{TWM}
T.W. McDaniel. Application of Landau-Lifshitz-Bloch dynamics to grain switching in heat-assisted magnetic recording, Journal of Applied Physics, 112(1): 013914, 2012.

\bibitem{LCE}
L.C. Evans. Partial differential equations. American Mathematical Soc., 2010.



\end{thebibliography}
\end{document}